\documentclass[10pt]{amsart}
\usepackage{amsmath}
\usepackage{amssymb}
\usepackage{enumerate}
\usepackage{amsbsy}
\usepackage{amsfonts}
\usepackage{color}

\newtheorem{thm}{Theorem}[section]
\newtheorem{lem}[thm]{Lemma}

\newtheorem{cor}[thm]{Corollary}

\numberwithin{equation}{section}

\begin{document}

	\title[Chern-Simons-Higgs System]{Almost critical regularity of non-abelian Chern-Simons-Higgs system in the Lorenz gauge}

	\author[Y. Cho]{Yonggeun Cho}
	\address{Department of Mathematics, and Institute of Pure and Applied Mathematics, Jeonbuk National University, Jeonju 54896, Republic of Korea}
	\email{changocho@jbnu.ac.kr}
	
    \author[S. Hong]{*Seokchang Hong}
    \address{Department of Mathematical Sciences, Seoul National University, Seoul 08826, Republic of Korea}
    \email{seokchangh11@snu.ac.kr}

	\thanks{*: Corresponding author}
	\thanks{2010 {\it Mathematics Subject Classification.} 35Q55, 35Q40.}
	\thanks{{\it Key words and phrases.} Chern-Simons-Higgs system, non-abelian gauge theory, almost critical regularity, null structure, bilinear estimates, failure of $C^2$ smoothness}
	
	\begin{abstract}
		In this paper we consider a Cauchy problem on the self-dual relativistic non-abelian Chern-Simons-Higgs model, which is the system of equations of $\mathfrak{su}(n)\, (n \ge 2)$-valued matter field $\phi$ and gauge field $A$. Based on the frequency localization as well as the null structure we show the local well-posedness in Sobolev space $H^{s+\frac12} \times H^s$ for $s>\frac14$. We also prove that the solution flow map $(\phi(0), A(0)) \mapsto (\phi(t), A(t))$ fails to be $C^2$ at the origin of $H^s \times H^\sigma$ when $\sigma < \frac14$ regardless of $s \in \mathbb R$. This means the regularity $H^s$, $s>\frac14$ is almost critical.
\end{abstract}

		\maketitle

\section{Introduction}


The Chern-Simons theory effectively describes $1+2$ dimensional physical phenomena in condensed matter physics.
In general, particles interacting via the Chern-Simons gauge acquire fractional statistics, which play a role in the fractional quantum Hall effect and also in high temperature superconductivity \cite{bajack}. After Chern and Simons first introduced geometric invariants \cite{chernsimons}, Chern-Simons gauge theory gets a lot of interest in physicists and mathematicians. For instance see \cite{chaeoh, dunne, dunne2} and references therein.

Recently, many mathematicians have studied various dispersive partial differential equations coupled with Chern-Simons gauges, especially, Chern-Simons-Dirac system (CSD) and Chern-Simons-Higgs system (CSH) under several gauge conditions; temporal gauge $A_0=0$, Coulomb gauge $\partial^jA_j=0$, and Lorenz gauge $\partial^\mu A_\mu=0$. It is well-known that in Lorenz gauge, (CSD) is rewritten as a system of nonlinear wave equation coupled with Dirac equation. The authors of \cite{huh, huh2, huhoh, pech} studied the low regularity solutions to (CSD) under the Lorenz gauge condition for which the Sobolev space $H^{\frac14}$ is expected to be the critical space in the sense of well-posedness \cite{oka}. On the other hand, (CSD) becomes a cubic Dirac equation with an elliptic structure in the Coulomb gauge. A number of results on low regularity solutions to (CSD) in the Coulomb gauge have appeared in \cite{boucanma, oka}. 

At the same time, the Higgs models in $1+2$ dimensions have attracted attention of mathematicians as well as physicists. A lot of works have been devoted to the low regularity theory of both abelian and non-abelian (CSH). See \cite{bou, chen, huh, huhoh, oh, yuan, yuan2} and references therein. 
However, despite the effort, progress has been made slow on the mathematical analysis of (CSH), which is due to the complexity of structure. The sharpness of regularity has not been known so far especially about non-abelian (CSH). In this paper, in the basis of antecedent works we pursue a rigorous well-posedness theory and make an endeavor to provide an almost optimal regularity on the full model of self-dual relativistic non-abelian Chern-Simons-Higgs system in the Lorenz gauge.

Let us begin with the mathematical setup for the non-abelian Chern-Simons-Higgs system. Let $G$ be a compact Lie group and $\mathfrak g$ its Lie algebra. For the sake of simplicity, we shall assume $ G = SU(n,\mathbb C),\ n \ge 2$ (the group of unitary matrices of determinant one). Then $\mathfrak g=\mathfrak{su}(n,\mathbb C)$ is the algebra of trace-free skew-Hermitian matrices whose infinitesimal generators are denoted by $T^a$ $(a = 1, 2, \cdots, n^2-1)$, and are traceless Hermitian matrices. (For example, $\mathfrak g = \mathfrak{su}(2,\mathbb C)$, $T^a$ is chosen to be  Pauli matrices; $i\sigma^a$, $a = 1,2,3$.)

For a given $\mathfrak g$-valued gauge field $A$, the component $A_\mu$ is written as $A_\mu(t, x) = A_{\mu, a}(t, x)T^a \ (a = 1,  \cdots n^2-1)$. We then define the curvature $F = dA + [A, A]$. More explicitly, given $A_\mu :\mathbb R^{1+2}\rightarrow\mathfrak g$, we define $F_{\mu\nu}$ by
$$
F_{\mu\nu} = \partial_\mu A_\nu - \partial_\nu A_\mu + [A_\mu, A_\nu].
$$
The associated covariant derivative is denoted by $\mathcal{D}_{\mu} = \partial_\mu+[A_\mu,\cdot]$. The matter field $\phi$ is also assumed to be $\mathfrak{su}(n)$-valued function and thus $\phi = \phi_aT^a$ because the most natural and interesting physical case seems to be with the matter fields and gauge fields in the same Lie algebra representation \cite{dunne, dunne2}.
Throughout this paper, we adopt the Einstein summation convention, where Greek indices refer to 0,1,2 and Latin indices $j, k$ refer to 1,2.
Indices are raised or lowered with respect to the Minkowski metric $\eta$ with signature $(+,-,-)$.

The Lagrangian density of the $1 + 2$ dimensional non-abelian relativistic Chern-Simons-Higgs system is defined by
$$
\mathcal L = -\frac\kappa2\epsilon^{\mu\nu\alpha}{\rm Tr}(\partial_\mu A_\nu A_\alpha+\frac23A_\mu A_\nu A_\alpha)+{\rm Tr}((\mathcal{D}_{\mu}\phi)^\dagger(\mathcal D^\mu\phi))-V(\phi,\phi^\dagger),
$$
where $V(\phi,\phi^\dagger)$ is the Higgs potential given by
$$
V(\phi,\phi^\dagger) = \frac{1}{\kappa^2}{\rm Tr}\big(([[\phi,\phi^\dagger],\phi]-v^2\phi)^\dagger([[\phi,\phi^\dagger],\phi]-v^2\phi)\big).
$$
Higgs potential is of sixth order and self-dual form, which is designed for a lower bound of energy. The constant $v > 0$ measures either the scale of the broken symmetry or the subcritical temperature of the system \cite{yang}. The $\epsilon^{\mu\nu\alpha}$ is the totally skew-symmetric tensor with $\epsilon^{012}=1$. ${\rm Tr}A$ and $A^\dagger$ denote the trace and $(\overline{A})^t$ the complex conjugate transpose of a matrix $A$, respectively. $[A, B] = AB - BA$ is the matrix commutator.

The Euler-Lagrange equation of the above Lagrangian density is
\begin{align}\label{el-csh}
\left\{
\begin{array}{l}
\mathcal{D}_{\mu}\mathcal D^\mu\phi = -\mathcal V(\phi,\phi^\dagger),\\
 F_{\mu\nu} = \epsilon_{\mu\nu\alpha}J^\alpha ,
\end{array}\right.\end{align}
where $J^\mu$ is defined by $J^\mu = [\phi^\dagger, \mathcal D^\mu\phi] - [(\mathcal D^\mu\phi)^\dagger,\phi]$ and $\mathcal V(\phi,\phi^\dagger)$ is derived by $\mathcal V(\phi,\phi^\dagger)=\dfrac{\partial V(\phi,\phi^\dagger)}{\partial\phi^\dagger}$. For simplicity, we assumed that the coupling constant $\kappa$ in front of $F_{\mu\nu}$ is $1$ in this paper. The potential $\mathcal V$ consists of linear, cubic, and quintic terms of $\phi$ and $\phi^\dagger$. (For details, see Appendix below.) In particular, it contains linear term $-2v^4\phi$ which contributes as a Higgs mass $m = \sqrt{2}v^2$ and gives a relativistic nature to \eqref{el-csh}. The initial data set for the system comprises  $(f, g, a_0, a_1, a_2)$, where $(\phi, \partial_t\phi)(0, x) = (f(x), g(x))$ and $A_{\mu}(0) = a_\mu$.

Now we take $\partial^\mu$ of the second equation in \eqref{el-csh} and use the Lorenz gauge $\partial^\mu A_\mu = 0$. Then we arrive at the following system of wave equations which describes the time evolution of the fields $A_\mu, \phi$.
\begin{align}\label{csh}
\left\{
\begin{array}{l}
\square\phi=-2\left[A^{\mu}, \partial_{\mu} \phi\right]-\left[A_{\mu},\left[A^{\mu}, \phi\right]\right] - \mathcal V(\phi,\phi^\dagger), \\
\square A_{\mu}=\left[\partial^{\nu} A_{\mu}, A_{\nu}\right]-\epsilon_{\mu \nu \alpha}\left(Q^{\nu \alpha}\left(\phi^{\dagger}, \phi\right)+Q^{\nu \alpha}\left(\phi, \phi^{\dagger}\right)\right)  -\epsilon_{\mu \nu \alpha}\partial^{\nu}\left(\left[\phi^{\dagger},\left[A^{\alpha}, \phi\right]\right]-\big[\left[A^{\alpha}, \phi\right]^{\dagger}, \phi\big]\right) \\
 (\phi, \partial_t \phi)(0) = (f, g), \quad A_\mu(0) = a_\mu,\\
 \partial_t A_0(0) = -\partial^j a_j, \quad \partial_tA_j(0) = \partial_j a_0 - [a_0, a_j] + \epsilon_{0jk}([f^\dagger, \partial^k f] - [(\partial^k f)^\dagger, f]),
\end{array}\right.\end{align}
where $Q_{\alpha\beta}(u, v) = \partial_\alpha u\partial_\beta v-\partial_\beta u\partial_\alpha v$. Note that due to the Lorenz gauge condition the initial data should satisfy the following constraint equation:
\begin{align}\label{constraint}
\partial_1a_2 - \partial_2a_1 + [a_1, a_2] = [f^\dagger, gf + [a^0, f]] - [(gf)^\dagger + [a^0, f]^\dagger, f].
\end{align}
We state our result on the local well-posedness:

\begin{thm}\label{lwp}
Let $s>\frac14$. Suppose that $(f, g) \in H^{s+\frac12} \times H^{s-\frac12}$, $a_\mu\in H^s$ and they satisfy \eqref{constraint}. Then the \eqref{csh} is locally well-posed in $H^{s+\frac12} \times H^{s-\frac12} \times H^{s}$. That is, there exists   $T = T(f, g, a_\mu,m) > 0$ such that there exist unique solution $(\phi, \partial_t \phi, A_\mu) \in C((-T,T);H^{s+\frac12} \times H^{s-\frac12} \times H^s)$ of \eqref{csh}, which depends continuously on the initial data.
\end{thm}

Here, $H^s$ is the usual inhomogeneous Sobolev space whose norm is given by $$\|f\|_{H^s} = (\sum_{N:\,\rm dyadic}(N^s\|P_{|\xi|\sim N}f\|_{L^2})^2)^\frac12,$$ where $P_{|\xi|\sim N}$ is the Littlewood-Paley projection on $\{ \xi\in\mathbb R^2 : |\xi|\sim N \}$. Instead of applying global estimates of \cite{danfoselb}, we make fully use of localization of space-time Fourier side. Thanks to the dyadic decomposition of space-time frequencies, we gain a {\it lower} regularity well-posedness, which will turn out to be sharp, than obtained in \cite{yuan2}. We exploit the null structure hidden in \eqref{csh} as \cite{yuan2}. In fact, this null structure has a similar form as Yang-Mills equation in $1+3$ dimensions introduced in \cite{selbtes, tes}.


Let us now deal with the smoothness of the flow map $(\phi(0), A(0)) \mapsto (\phi(t), A(t))$. Since the nonlinearity is algebraic, one may expect the flow will be smooth in local time if the problem is well-posed. However, such smoothness can be shown to fail when the initial data of $\phi$, $A$ are rougher than in $H^\frac12$, $H^\frac14$, respectively, which can be stated as follows.
\begin{thm}\label{fail-smooth}
Let $s \in \mathbb R$, $\sigma < \frac14$, and $T > 0$. Then the flow map of $(\phi(0), A(0)) \mapsto (\phi(t), A(t))$ from $H^{s} \times H^\sigma$ to $C([-T, T]; H^{s} \times H^\sigma)$ cannot be $C^2$ at the origin. Furthermore, if $s < \frac12$, $\sigma \in \mathbb R$, then the flow map cannot be $C^3$ at the origin.
\end{thm}
We proceed the proof by the argument of Knapp type example as in \cite{mst}. We investigate carefully $\mathfrak{su}(n)$-valued initial data which guarantee that a resonance of phases occurs and hence the Fourier transforms of matter field and gauge field have significant lower bounds, which enable us to get the necessary condition $s \ge \frac12$ or $\sigma \ge \frac14$ for the smoothness. In view of Theorem \ref{fail-smooth}, the LWP of Theorem \ref{lwp} is very sharp since $H^s$ is a proper subspace of $H^\frac14$ for $s>\frac14$. For the present we could not have filled the regularity or failure of smoothness fully in $H^\frac14$ by a technical reason. However, the problem will be hopefully resolved in the near future.


We end this section with the introduction of notations and organization of this paper.

\noindent\textbf{Notations.} Here we give some notations used throughout this paper. Since we only use $L_{t, x}^{2}$ norm, by $\|F\|$ we abbreviate $\|F\|_{L^{2}_{t,x}} := \sum_a \|F_a\|_{L^{2}_{t,x}}$ for $F = F_aT^a$. As usual different positive constants independent on dyadic numbers such as $N$ and $L$ are denoted by the same letter $C$, if not specified. $A \lesssim B$ and $A \gtrsim B$ means that $A \le CB$ and
$A \ge C^{-1}B$, respectively for some $C>0$. $A \sim B$ means that $A \lesssim B$ and $A \gtrsim B$.

The spatial Fourier transform and space-time Fourier transform on $\mathbb{R}^{2}$ and $\mathbb R^{1+2}$ are defined by
$$
\widehat{f}(\xi)=\int_{\mathbb{R}^{2}}e^{-ix\cdot\xi}f(x)dx,\quad \widetilde{u}(X)=\int_{\mathbb R^{1+2}}e^{-i(t\tau+x\cdot\xi)}u(t,x)dtdx,
$$
where $\tau\in\mathbb{R}$, $\xi\in\mathbb{R}^{2}$, and $X = (\tau,\xi)\in\mathbb R^{1+2}$. Also we denote $\mathcal{F}(u)=\widetilde{u}$. Then we define space-time Fourier projection operator $P_E$ by $\widetilde{P_E u}(\tau, \xi) = \chi_E \widetilde u(\tau, \xi)$, for $E \subset \mathbb R^{1+2}$. We define spatial Fourier projection operator, similarly. For example, $P_{|\xi|\sim N}$ is the Littlewood-Paley projection on $\{ \xi\in\mathbb R^2 : |\xi|\sim N \}$.

Since we prefer to use the differential operator $|\nabla|$ rather than $-i\nabla$, for the sake of simplicity, we put $D := |\nabla|$ whose symbol is $|\xi|$.

For brevity, we denote the maximum, median, and minimum of $N_0,\ N_1,\ N_2$ by
$$
N_{\max}^{012} = \max(N_0,N_1,N_2), \quad N_{\rm med}^{012} = {\rm med}(N_0,N_1,N_2), \quad N_{\min}^{012} = \min(N_0,N_1,N_2).
$$

\noindent\textbf{Organization.} In Section 2, we introduce the decomposition of d'Alembertian and $X^{s,b}$ space. Section 3 is devoted to the description on our main techniques; 2D wave type bilinear estimates and null structure. In Section 4, we give crucial estimates to prove the local well-posedness of \eqref{csh}. Here, we observe that the estimates of commutator of $\mathfrak{su}(n)$-valued functions are reduced to the nonlinear estimates of $\mathbb C$-valued functions. Then Section 5,6,7 are on the proof of bilinear estimates, trilinear estimates, and estimates of Higgs potential, respectively. In Section 8 we show the failure of smoothness.


\section{Preliminaries}

\subsection{Decomposition of d'Alembertian}
We use the standard transform given by $(\phi,\partial_t\phi)\rightarrow(\phi_+,\phi_-)$ and $(A_\mu,\partial_tA_\mu)\rightarrow(A_{\mu,+},A_{\mu,-})$ with
$$
\phi_{\pm} = \frac12\left(\phi\pm\frac{1}{iD}\partial_t\phi\right), \quad A_{\mu,\pm} = \frac12\left(A_\mu\pm\frac{1}{iD}\partial_tA_\mu\right).
$$
Then the system \eqref{csh} transforms to
\begin{align}\label{csh-decom}
\left\{
\begin{array}{l}
(i\partial_t\pm D)\phi_\pm=\pm\frac{1}{2D}\left(-2\left[A^{\mu}, \partial_{\mu} \phi\right]-\left[A_{\mu},\left[A^{\mu}, \phi\right]\right]-\mathcal V(\phi,\phi^\dagger)\right) \\
(i\partial_t\pm D) A_{\mu,\pm}=\pm\frac{1}{2D}\big(\left[\partial^{\nu} A_{\mu}, A_{\nu}\right]-\epsilon_{\mu \nu \alpha}\left(Q^{\nu \alpha}\left(\phi^{\dagger}, \phi\right)+Q^{\nu \alpha}\left(\phi, \phi^{\dagger}\right)\right)\big) \\
 \qquad \quad \mp\frac{\epsilon_{\mu \nu \alpha}}{2D}\partial^{\nu}\left(\left[\phi^{\dagger},\left[A^{\alpha}, \phi\right]\right]-\big[\left[A^{\alpha}, \phi\right]^{\dagger}, \phi\big]\right).
\end{array}\right.\end{align}

\subsection{Function spaces}
For dyadic number $N\ge1$, and $L$, we define the set
$$
K_{N,L}^\pm = \{ (\tau,\xi)\in\mathbb R^{1+2} : |\xi|\sim N,\quad |\tau\pm|\xi||\sim L \}.
$$
Then we introduce Bourgain space given by
$$
X^{s,b}_\pm = \left\{ u\in L^2 : \|u\|_{X^{s,b}_\pm}= \||\xi|^s|\tau\pm|\xi||^b\widetilde{u}(\tau,\xi)\|<\infty \right\}.
$$
We make the use of fully dyadic decompostion of spacetime-Fourier sides, and hence reformulate the $X^{s,b}_\pm$-norm by Littlewood-Paley decomposition using $K_{N,L}^\pm$:
$$
\|u\|_{X^{s,b}_\pm} = \left(\sum_{N,L}(N^sL^b\|P_{K_{N,L}^\pm}u\|)^2\right)^\frac12.
$$
Since we are only concerned with local time existence $T\le1$ throughout this paper, it is convenient to utilize our function space in the local time setting. Hence we introduce the restriction space.
The time-slab which is the subset of $\mathbb R^{1+2}$ is given by
$$
S_{T}=(-T,T)\times\mathbb{R}^{2}.
$$
We let $X^{s,b}_\pm(S_T)$ be the restriction space to the time-slab $S_T$. Recall the following embedding property for $b>\frac12$:
\begin{equation}
X_{\pm}^{s, b}\left(S_{T}\right) \hookrightarrow C\left([-T, T] ; H^{s}\right).
\end{equation}
Furthermore, it is the well-known fact that given linear initial value problem:
$$
(i\partial_t\pm D)v = G\in X^{s,b-1+\epsilon}_\pm(S_T),\quad v(0)\in H^s,
$$
for $s\in\mathbb R$, $b>\frac12$, and $0<\epsilon\ll1$, it has a unique solution satisfying
\begin{equation}\label{emb}
\|v\|_{X^{s,b}_\pm(S_T)} \lesssim \|v(0)\|_{H^s}+T^\epsilon\|G\|_{X^{s,b-1+\epsilon}_\pm(S_T)},\quad T<1.
\end{equation}

\section{Bilinear estimates and Null structure}
\subsection{Bilinear estimates}
For dyadic $N,L\ge 1$, let us invoke that
$$
K_{N,L}^{\pm} = \{ (\tau,\xi)\in\mathbb R^{1+2} : |\xi|\sim N,\;\; |\tau \pm |\xi| |\sim L \}.
$$
Now we introduce the key ingredient to handle the nonlinear terms in \eqref{csh}.
\begin{thm}[Theorem 2.1 of \cite{selb}]\label{bilinear-est}
For all $u_{1},u_{2}\in L^{2}_{t,x}(\mathbb{R}^{1+2})$ such that $\widetilde{u_{j}}$ is supported in $K_{N_{j},L_{j}}^{\pm_{j}}$, the estimate
$$
\|P_{K_{N_{0},L_{0}}^{\pm_{0}}}(u_{1}\overline{u_{2}})\| \le C\|u_{1}\|\|u_{2}\|
$$
holds with
\begin{eqnarray}
C & \sim & (N_{\min}^{012}L_{\min}^{12})^{\frac12}(N_{\min}^{12}L_{\max}^{12})^{\frac14},\label{bi-selb-1} \\
C & \sim & (N_{\min}^{012}L_{\min}^{0j})^{\frac12}(N_{\min}^{0j}L_{\max}^{0j})^{\frac14}, \quad j = 1, 2,\label{bi-selb-2} \\
C & \sim & ((N_{\min}^{012})^{2}L_{\min}^{012})^{\frac12}\label{bi-selb-3}
\end{eqnarray}
regardless of the choices of signs $\pm_{j}$.
\end{thm}

\subsection{Bilinear interaction}
The space-time Fourier transform of the product $\phi_{2}^{\dagger}\phi_{1}$ of two $\mathfrak g$-valued fields $\phi_{1}$ and $\phi_{2}$ is written as
$$
\widetilde{\phi_{2}^{\dagger}\phi_{1}}(X_{0})=\int_{X_{0} = X_{1} - X_{2}}\widetilde{\phi_{2}}^{\dagger}(X_{2})\widetilde{\phi_{1}}(X_{1})\, dX_{1}\,dX_{2},
$$
where $\phi^{\dagger}$ is the transpose of complex conjugate of $\phi$. Here the relation between $X_{1}$ and $X_{2}$ in the convolution integral of fields is given by $X_{0} = X_{1} - X_{2}$ so called bilinear interaction. This is also the case for the product of two complex scalar fields. 

The following lemma is on the bilinear interaction.
\begin{lem}[Lemma 2.2 of \cite{selb}]\label{bi-int}
Given a bilinear interaction $(X_{0},X_{1},X_{2})$ with $\xi_{j} \neq 0$, and signs $(\pm_{0},\pm_{1},\pm_{2})$, let $h_{j}=\tau_{j}\pm_{j}|\xi_{j}|$ and $\theta_{12} = |\angle(\pm_{1}\xi_{1},\pm_{2}\xi_{2})|$. Then we have
$$
\max(|h_{0}|,|h_{1}|,|h_{2}|) \gtrsim \min(|\xi_{1}|,|\xi_{2}|)\theta_{12}^{2}.
$$
Moreover, we either have
$$
|\xi_0|\ll |\xi_1|\sim |\xi_2|,\quad {\rm and} \quad \pm_1\neq\pm_2,
$$
in which case
$$
\theta_{12}\sim 1 \quad {\rm and} \quad \max(|h_0|,|h_1|,|h_2|)\gtrsim \min(|\xi_1|,|\xi_2|),
$$
or else we have
$$
\max(|h_0|,|h_1|,|h_2|)\gtrsim \frac{|\xi_1||\xi_2|}{|\xi_0|}\theta_{12}^2.
$$
\end{lem}

\subsection{Null structure}


While proving the local well-posedness of \eqref{csh}, we must encounter multilinear estimates such as bilinear and trilinear estimates. Since we make use of duality argument and Cauchy-Schwarz inequality, we essentially treat only bilinear forms of wave type. Then the most serious case resulting in resonance interaction occurs when two input-waves are collinear. However, once this bilinear form possesses cancellation property so called null structure, we can expect better estimates \cite{klama}. 

Before discussing the null structure, we note that the spatial part of vector potential $\mathbf A = (A_1, A_2)$ can be split into divergence-free and curl-free parts:
$$
\mathbf A = \mathbf A^{\rm df} + \mathbf A^{\rm cf},
$$
where
\begin{align*}
A^{\rm df}_j &= (-\Delta)^{-1}\epsilon_{0jk}\partial^k(\epsilon^{0lm}\partial_lA_m) \\
A^{\rm cf}_j &= -(-\Delta)^{-1}\partial_j\partial^kA_k.
\end{align*}
Also we define the Riesz transform given by
$$
R_j = D^{-1}\partial_j = \frac{\partial_j}{D}.
$$
Now we introduce the standard null forms:
\begin{align*}
Q_0(u,v) &= \partial_\alpha u\partial^\alpha v \\
Q_{\alpha\beta}(u,v) &= \partial_\alpha u\partial_\beta v-\partial_\beta u\partial_\alpha v.
\end{align*}
Then we define a commutator version of null forms by
\begin{align*}
Q_0[u,v] &= [\partial_\alpha u,\partial^\alpha v] \\
Q_{\alpha\beta}[u,v] &= [\partial_\alpha u,\partial_\beta v]-[\partial_\beta u,\partial_\alpha v].
\end{align*}
Here we give some remark on commutator version of null forms. For $\mathfrak{su}(n)$-valued functions $u$ and $v$, we write $u=u_aT^a$ and $v=v_bT^b$, where $a,b=1,2,\cdots,n^2-1$ and $u_a,v_b$ are smooth scalar functions. Then there holds
\begin{align*}
Q_{\alpha\beta}[u,v] &= [\partial_\alpha u,\partial_\beta v] - [\partial_\beta u,\partial_\alpha v] \\
&= [ \partial_\alpha u_aT^a,\partial_\beta v_bT^b ] - [\partial_\beta u_aT^a,\partial_\alpha v_bT^b] \\
&= \partial_\alpha u_a\partial_\beta v_b[T^a,T^b] - \partial_\beta u_a\partial_\alpha v_b[T^a,T^b] \\
&= (\partial_\alpha u_a\partial_\beta v_b-\partial_\beta u_a\partial_\alpha v_b)[T^a,T^b]\\
&= Q_{\alpha\beta}(u_a,v_b)[T^a,T^b]\\
&= Q_{\alpha\beta}(u_a,v_b)if^{ab}_{\phantom{ab}c}T^c.
\end{align*}
For the last equality see Appendix below. Then for a function space $\mathcal X(\mathfrak{su}(n))$ defined by the functions with value in $\mathfrak{su}(n)$, we observe that
$$
\|Q_{\alpha\beta}[u,v]\|_{\mathcal X(\mathfrak{su}(n))} = \sum_{c}\|Q_{\alpha\beta}(u_a,v_b)\|_{\mathcal X(\mathbb C)}|f^{ab}_{\phantom{ab}c}|.
$$
Hence we conclude that the $\mathcal X(\mathfrak{su}(n))$ norm of commutator version of null forms is reduced to the $\mathcal X(\mathbb{C})$ norm of null forms of scalar functions.

The following lemma is on null structure hidden in \eqref{csh}.
\begin{lem}\label{null-csh}
In the Lorenz gauge, we have the following identity:
$$
[A^\mu,\partial_\mu\phi] = \frac12\epsilon^{0jk}\epsilon_{0lm}Q_{jk}[D^{-1}R^lA^m,\phi]-Q_{j0}[R^j(D^{-1}A_0),\phi].
$$
\end{lem}
\begin{proof}
First, we note that
$$
A^\mu\partial_\mu\phi = A_0\partial_t\phi+\mathbf A^{\rm cf}\cdot\nabla\phi+\mathbf A^{\rm df}\cdot\nabla\phi.
$$	
By Lorenz gauge condition: $\partial^kA_k=\partial_tA_0$, we get
\begin{align*}
\mathbf A^{\rm cf}\cdot\nabla\phi &= -(-\Delta)^{-1}\partial_j\partial^kA_k\partial_j\phi \\
&= D^{-2}\partial^j(\partial_tA_0)\partial_j\phi\\
&= \partial_tR^j(D^{-1}A_0)\partial_j\phi.
\end{align*}
Also we have
\begin{align*}
A_0\partial_t\phi &= -D^{-2}\partial_j\partial^jA_0\partial_t\phi\\
&= -\partial_jR^j(D^{-1}A_0)\partial_t\phi,
\end{align*}
and hence we have
$$
A_0\partial_t+\mathbf A^{\rm cf}\cdot\nabla\phi = -Q_{j0}(R^jD^{-1}A_0,\phi).
$$
Next, we see that
\begin{align*}
\mathbf A^{\rm df}\cdot\nabla\phi &= (-\Delta)^{-1}\epsilon^{0jk}\epsilon_{0lm}\partial_k\partial^lA^m\partial_j\phi \\
&= -\epsilon^{0jk}\epsilon_{0lm}\partial_k(D^{-1}R^lA_m)\partial_j\phi \\
&= \frac12\epsilon^{0jk}\epsilon_{0lm}Q_{jk}(D^{-1}R^lA^m,\phi).	
\end{align*}
We can treat $\partial_\mu A^\mu$ similarly and hence completes the proof.
\end{proof}
We have the following corollary by Lemma \ref{null-csh}.
\begin{cor}\label{null-csh2}
In the Lorenz gauge, we have the following identity.
$$
[\partial^\nu A_\mu,A_\nu] = -\frac12\epsilon^{0jk}\epsilon_{0lm}Q_{jk}[D^{-1}R^lA^m,A_\mu]+Q_{j0}[R^j(D^{-1}A_0),A_\mu].
$$
\end{cor}


\section{Proof of Local Well-Posedness}

\subsection{Picard's iterates}

To prove (LWP) of \eqref{csh}, we construct Picard's iterates and follow the contraction principle. Indeed, by \eqref{csh-decom}, we obtain the following integral equation for $\phi_{\pm}$ and $A_{\mu,\pm}$ respectively:
\begin{align}
&\phi_{\pm}(t) = \phi^{\rm hom}_{\pm}(t) \pm i\int_0^t\frac{e^{\mp i(t-t')D}}{2iD}\mathcal M(\phi,A)(t')\,dt'\label{phi-int} \\
&A_{\mu,\pm} = A^{\rm hom}_{\mu,\pm}(t) \pm i\int_0^t\frac{e^{\mp i(t-t')D}}{2iD}\mathcal N_\mu(\phi,A)(t')\,dt'\label{a-int}
\end{align}
where
$$
\psi^{\rm hom}_{\pm}(t) = \frac12e^{\mp itD}\left(\psi(0,x)\mp\frac{1}{iD}\partial_0\psi(0,x)\right).
$$
Here, $\mathcal M$ and $\mathcal N_\mu$ are the (LHS) of the first two equations of \eqref{csh}.
To prove that our Picard's iteration converges, it suffices to show the following estimates:


\begin{align}
&\|\phi^{\rm hom}_{\pm}\|_{X^{s+\frac12,b}_\pm(S_T)}\lesssim \|f\|_{H^{s+\frac12}}+\|g\|_{H^{s-\frac12}},\label{hom1} \\
&\|A^{\rm hom}_{\mu,\pm}\|_{X^{s,b}_\pm(S_T)} \lesssim \sum_{\nu=0}^2\|a_\nu\|_{H^{s}}+\|a_0\|_{H^{s}}\sum_{\nu=0}^2\|a_\nu\|_{H^{s}}+\|f\|_{H^{s+\frac12}}^2,\label{hom2} \\
&\|\mathcal M(\phi,A)\|_{X^{s-\frac12,b-1}_\pm(S_T)} \lesssim \mathfrak S(1+\mathfrak S+\mathfrak S^2+\mathfrak S^4),\label{nonl1} \\
&\|\mathcal N_\mu(\phi,A)\|_{X^{s-1,b-1}_\pm(S_T)} \lesssim \mathfrak S(1+\mathfrak S+\mathfrak S^2),\label{nonl2}
\end{align}
where $\mathfrak S=\sum_\pm(\|A_\pm\|_{X^{s,b}_\pm(S_T)}+\|\phi_\pm\|_{X^{s+\frac12,b}_\pm(S_T)})$.

\subsection{Estimates of $\phi^{\rm hom}_\pm$ and $A^{\rm hom}_{\mu,\pm}$}
Since $f\in H^{s+\frac12}$ and $g\in H^{s-\frac12}$, we see that
$$
\|\phi^{\rm hom}_\pm\|_{X^{s+\frac12,b}_\pm} \lesssim \left\|f\mp\frac{1}{iD}g\right\|_{H^{s+\frac12}}\le \|f\|_{H^{s+\frac12}}+\left\|\frac{1}{iD}g\right\|_{H^{s+\frac12}} \lesssim \|f\|_{H^{s+\frac12}}+\|g\|_{H^{s-\frac12}}.
$$
On the other hand, $A^{\rm hom}_{0}$ and $A^{\rm hom}_{j}$ are given by
\begin{align*}
A_0^{\rm hom} &= \sum_\pm\frac12e^{\mp itD}\left(a_0\pm\frac{\partial^j}{iD}a_j\right), \\
A_j^{\rm hom} &= \sum_\pm\frac12e^{\mp itD}\left(a_j\mp\frac{1}{iD}\left(\partial_ja_0-[a_0,a_j]+\epsilon_{0jk}([f^\dagger,\partial^kf]-[(\partial^kf)^\dagger,f])\right)\right).
\end{align*}
Since Riesz transform is a bounded operator in $L^2$,
$$
\|A_0^{\rm hom}\|_{X^{s,b}_\pm} \lesssim \|a_0\|_{H^{s}}+\left\|\frac{\partial^j}{iD}a_j\right\|_{H^{s}} \lesssim \|a_0\|_{H^{s}}+\sum_{j=1}^2\|a_j\|_{H^{s}}.
$$
To treat $A^{\rm hom}_j$ we recall Bernstein's inequality,
\begin{align}\label{Bernstein-ineq}
\|P_{|\xi|\sim N}f\|_{L^p(\mathbb R^d)} &\lesssim N^{d(\frac1q-\frac1p)}\|P_{|\xi|\sim N}f\|_{L^q(\mathbb R^d)},
\end{align}
for $q<p\le\infty$.
Then we have
\begin{align*}
\left\|\frac{1}{iD}(f\partial^kf)\right\|_{H^{s}}^2 &\lesssim \|f\partial^kf\|_{H^{s-1}}^2 = \sum_{N:{\rm dyadic}}(N^{s-1}\|P_N(f\partial^kf)\|)^2 \\
&\le \sum_{N}(N^{s-1}\|P_Nf\|_{L^4}\|P_N\partial^kf\|_{L^4})^2 \\
&\lesssim \sum_{N}(N^{s+1}\|P_Nf\|\|P_Nf\|)^2 \\
&\lesssim \sum_N(N^{s+\frac12}\|P_Nf\|)^2\sum_N(N^{s+\frac12}\|P_Nf\|)^2 \\
&\lesssim \|f\|_{H^{s+\frac12}}^4,
\end{align*}
and similarly,
$$
\left\|\frac{1}{iD}(a_0a_j)\right\|_{H^{s}} \lesssim \|a_0a_j\|_{H^{s-1}} \lesssim \|a_0\|_{H^{s}}\|a_j\|_{H^{s}}.
$$
Therefore,
$$
\|A_j^{\rm hom}\|_{X^{s,b}_\pm} \lesssim \|a_0\|_{H^{s}}+\|a_j\|_{H^{s}}+\|a_0\|_{H^{s}}\|a_j\|_{H^{s}}+\|f\|^2_{H^{s+\frac12}}.
$$
This proves \eqref{hom1} and \eqref{hom2}.

\subsection{Reduction step}
Now we reduce the nonlinear estimates \eqref{nonl1} and \eqref{nonl2} of $\mathfrak{su}(n)$-valued functions to those of scalar functions. We have already observed in Section 3 that the matrix structure in null forms plays no crucial role in the estimates. Also, we claim that estimates of cubic and quintic terms of $A^\mu,\phi$ are reduced to the nonlinear estimates of scalar functions. For example, let us consider the cubic terms $[A^\mu_{\pm_1},[A_{\mu,\pm_2},\phi_{\pm_3}]]$. We write $A^\mu_{\pm_1}=A^\mu_{\pm_1,a}T^a$, $A_{\mu,\pm_2}=A_{\mu,\pm_2,b}T^b$, and $\phi_{\pm_3}=\phi_{\pm_3,c}T^c$. Then for a Lebesgue or Sobolev space $\mathcal X(\mathfrak{su}(n))$ of $\mathfrak{su}(n)$-valued functions, the norm of $[A^\mu_{\pm_1},[A_{\mu,\pm_2},\phi_{\pm_3}]]$ is given by
$$
\|[A^\mu_{\pm_1},[A_{\mu,\pm_2},\phi_{\pm_3}]]\|_{\mathcal X(\mathfrak{su}(n))} = \sum_{e}\|A^\mu_{\pm_1,a}A_{\mu,\pm_2,b}\phi_{\pm_3,c}\|_{\mathcal X(\mathbb C)}|f_{\phantom{bd}d}^{bc}||f_{\phantom{ad}e}^{ad}|,
$$
where $\|\cdot\|_{\mathcal X (\mathbb C)}$ is the norm of space $\mathcal X$ of complex-valued functions, and thus the norm of cubic terms of $\mathfrak{su}(n)$-valued fields is reduced to the norm of cubic terms of scalar fields.

In Appendix, we shall see that $a$-th component of $\mathcal V(\phi,\phi^\dagger)$ is given by
$$
\dfrac{\partial V(\phi,\phi^\dagger)}{\partial\phi_a^*} = 2\sum_{e}(f^{ab}_{\phantom{ab}d}f^{dc}_{\phantom{dc}e}\phi_b^*\phi_c^* + v^2\delta^{ea}) (f^{a'b'}_{\phantom{a'b'}d'}f^{d'c'}_{\phantom{d'c'}e}\phi_{a'}\phi_{b'}\phi_{c'} + v^2\phi_{e}).
$$
Then the $\mathcal X(\mathfrak{su}(n))$ norm of the quintic term in $\mathcal V$ is bounded by
$$
2\sum_{a, e}\|\phi_b^*\phi_c^*\phi_{a'}\phi_{b'}\phi_{c'}\|_{\mathcal X(\mathbb C)}|f^{ab}_{\phantom{ab}d}f^{dc}_{\phantom{dc}e}f^{a'b'}_{\phantom{a'b'}d'}f^{d'c'}_{\phantom{d'c'}e}|,
$$
and hence the norm of quintic terms of $\phi,\phi^\dagger$ is reduced to the $\mathcal X(\mathbb C)$ norm of quintic terms of scalar fields. Therefore, from now on, we consider the $\mathfrak{su}(n)$-valued functions $A^\mu,\phi$ as $\mathbb C$-valued functions.
Here we assume $s>\frac14$, $b>\frac12$. Since we are concerned with low regularity solution, we may assume $0<s-\frac14\ll1$ and $0<b-\frac12\ll 1$. Hence we write
$$
s=\frac14+\delta,\quad b=\frac12+\epsilon,
$$
where $0<100\epsilon<\delta\ll1$.
In the following three sections, we will focus on the proof of the above nonlinear estimates \eqref{nonl1} and \eqref{nonl2}.
\section{Bilinear estimates}

This section is devoted to the proof of the bilinear estimates appearing in Section 4. Since the Riesz transforms $R_i$ are bounded in the spaces under our consideration, these bilinear estimates can be reduced to the following:
\begin{align}
& \|Q_{jk}(D^{-1}A_{\pm_1},\phi_{\pm_2})\|_{X^{s-\frac12,b-1}_\pm(S_T)} \lesssim \|A_{\pm_1}\|_{X^{s,b}_{\pm_1}}\|\phi_{\pm_2}\|_{X^{s+\frac12,b}_{\pm_2}}\label{bi-aphi-jk}, \\
& \|Q_{j0}(D^{-1}A_{\pm_1},\phi_{\pm_2})\|_{X^{s-\frac12,b-1}_\pm(S_T)} \lesssim \|A_{\pm_1}\|_{X^{s,b}_{\pm_1}}\|\phi_{\pm_2}\|_{X^{s+\frac12,b}_{\pm_2}}\label{bi-aphi-j0}, \\
& \|Q_{jk}(D^{-1}A_{\pm_1},A_{\pm_2})\|_{X^{s-1,b-1}_\pm(S_T)} \lesssim \|A_{\pm_1}\|_{X^{s,b}_{\pm_1}}\|A_{\pm_2}\|_{X^{s,b}_{\pm_2}}\label{bi-aa-jk}, \\
& \|Q_{j0}(D^{-1}A_{\pm_1},A_{\pm_2})\|_{X^{s-1,b-1}_\pm(S_T)} \lesssim \|A_{\pm_1}\|_{X^{s,b}_{\pm_1}}\|A_{\pm_2}\|_{X^{s,b}_{\pm_2}}\label{bi-aa-j0}, \\
& \|Q_{jk}(\overline{\phi_{\pm_1}}, \phi_{\pm_2})\|_{X^{s-1,b-1}_\pm(S_T)} \lesssim \|\phi_{\pm_1}\|_{X^{s+\frac12,b}_{\pm_1}}\|\phi_{\pm_2}\|_{X^{s+\frac12,b}_{\pm_2}}\label{bi-phiphi-jk}, \\
& \|Q_{j0}(\overline{\phi_{\pm_1}}, \phi_{\pm_2})\|_{X^{s-1,b-1}_\pm(S_T)} \lesssim \|\phi_{\pm_1}\|_{X^{s+\frac12,b}_{\pm_1}}\|\phi_{\pm_2}\|_{X^{s+\frac12,b}_{\pm_2}}\label{bi-phiphi-j0}.
\end{align}
To prove the above bilinear estimates via null forms for functions $u$ and $v$, we recall the substitution:
\begin{align*}
& u = u_++u_-, \quad \partial_t u = iD(u_+-u_-),\\
& v = v_++v_-, \quad \partial_t v = iD(v_+-v_-).
\end{align*}
Then we have
\begin{align*}
& Q_{j0}(u,v) = \sum_{\pm_1,\pm_2}(\pm_11)(\pm_21)\left((\pm_1\partial_ju_{\pm_1})(\pm_2iDv_{\pm_2})-(\pm_1iDu_{\pm_1})(\pm_2\partial_jv_{\pm_2})\right),\\
& Q_{jk}(u,v) = \sum_{\pm_1,\pm_2}(\pm_11)(\pm_21)\left((\pm_1\partial_ju_{\pm_1})(\pm_2\partial_kv_{\pm_2})-(\pm_1\partial_ku_{\pm_1})(\pm_2\partial_jv_{\pm_2})\right),
\end{align*}
and their symbols are given by
\begin{align*}
& q_{j0}(\xi_1,\xi_2) = -\xi_{1,j}|\xi_2|+|\xi_1|\xi_{2,j},\\
& q_{jk}(\xi_1,\xi_2) = -\xi_{1,j}\xi_{2,k}+\xi_{1,k}\xi_{2,j}.
\end{align*}
For these symbols, we have the following estimates.
\begin{lem}\label{null-ineq}
For $\xi_1,\xi_2\in\mathbb R^2$ with $\xi_1,\xi_2\neq0$,
\begin{align*}
& |q_{j0}(\xi_1,\xi_2)|,|q_{jk}(\xi_1,\xi_2)| \lesssim |\xi_1||\xi_2||\angle(\xi_1,\xi_2)|.
\end{align*}
\end{lem}
\begin{proof}
See Lemma 5. of \cite{selbtes}. 
\end{proof}
By Lemma \ref{null-ineq}, it suffices to prove \eqref{bi-aphi-jk}, \eqref{bi-aa-jk}, and \eqref{bi-phiphi-jk}. Furthermore, the proof of \eqref{bi-aphi-jk} and \eqref{bi-aa-jk} is essentially same. Thus we focus on the proof of \eqref{bi-aphi-jk} and \eqref{bi-phiphi-jk}.

\subsection{Proof of \eqref{bi-aphi-jk}}

We write
\begin{align*}
\|Q_{jk}(D^{-1}A_{\pm_1},\phi_{\pm_2})\|_{X^{s-\frac12,b-1}_{\pm_0}}^2 &= \sum_{N_0,L_0}\left(N_0^{-\frac14+\delta}L_0^{-\frac12+\epsilon}\|P_{K_{N_0,L_0}^{\pm_0}}Q_{jk}(D^{-1}A_{\pm_1},\phi_{\pm_2})\|\right)^2 \\
&\lesssim \sum_{N_0,L_0}\left(N_0^{-\frac14+\delta}L_0^{-\frac12+\epsilon}\sum_{N_1,N_2}\sum_{L_1,L_2}\|P_{K_{N_0,L_0}^{\pm_0}}Q_{jk}(D^{-1}A^{\pm_1}_{N_1,L_1}\phi^{\pm_2}_{N_2,L_2}\|\right)^2 \\
&\lesssim \sum_{N_0,L_0}\left(N_0^{-\frac14+\delta}L_0^{-\frac12+\epsilon}\sum_{N_1,N_2}\sum_{L_1,L_2}N_2\theta_{12}\|P_{K_{N_0,L_0}^{\pm_0}}(A^{\pm_1}_{N_1,L_1}\phi^{\pm_2}_{N_2,L_2})\|\right)^2 \\
&\lesssim \sum_{N_0,L_0}\left(N_0^{-\frac14+\delta}L_0^{-\frac12+\epsilon}\sum_{N_1,N_2}\sum_{L_1,L_2}N_2\theta_{12}C_{N,L}^{012}\|A^{\pm_1}_{N_1,L_1}\|\|\phi^{\pm_2}_{N_2,L_2}\|\right)^2,
\end{align*}
where we write $A^{\pm_1}_{N_1,L_1}=P_{K_{N_1,L_1}^{\pm_1}}A_{\pm_1}$, $\phi^{\pm_2}_{N_2,L_2}=P_{K_{N_2,L_2}^{\pm_2}}\phi_{\pm_2}$ for brevity and $C_{N,L}^{012}$ is \eqref{bi-selb-1} or \eqref{bi-selb-2}.

Then it suffices to show that
\begin{align*}
\mathbf J^1 &:= \sum_{N_1,N_2}\sum_{L_1,L_2}N_0^{-\frac14+\delta}L_0^{-\frac12+\epsilon}N_2\theta_{12}C_{N,L}^{012}\|A^{\pm_1}_{N_1,L_1}\|\|\phi^{\pm_2}_{N_2,L_2}\| \\
&\lesssim \sum_{N_1,N_2}\sum_{L_1,L_2}N_0^\delta N_1^\frac14 N_2^\frac34 L_0^\epsilon(L_1L_2)^\frac12\|A^{\pm_1}_{N_1,L_1}\|\|\phi^{\pm_2}_{N_2,L_2}\|.
\end{align*}
To see this, we only consider $L_1\le L_2\ll L_0$. The general square summation by $L_1,L_2$ gives
\begin{align*}
\|Q_{jk}(D^{-1}A_{\pm_1},\phi_{\pm_2})\|_{X^{s-\frac12,b-1}_{\pm_0}}^2 &\lesssim \sum_{N_0,L_0}(N_0^\delta L_0^\epsilon\sum_{N_1,N_2} N_1^\frac14 N_2^\frac34 \|A^{\pm_1}_{N_1}\|_{X^{0,b}_{\pm_1}}\|\phi^{\pm_2}_{N_2}\|_{X^{0,b}_{\pm_2}})^2 \\
&\lesssim \sum_{N_0}(N_0^\delta \sum_{N_1,N_2} (N_{\min}^{12})^\epsilon N_1^\frac14 N_2^\frac34 \|A^{\pm_1}_{N_1}\|_{X^{0,b}_{\pm_1}}\|\phi^{\pm_2}_{N_2}\|_{X^{0,b}_{\pm_2}})^2.
\end{align*}
If $N_0\ll N_1\sim N_2$, then
\begin{align*}
\|Q_{jk}(D^{-1}A_{\pm_1},\phi_{\pm_2})\|_{X^{s-\frac12,b-1}_{\pm_0}}^2 &\lesssim \sum_{N_0\ge1}(N_0^\delta N_0^{\epsilon-2\delta}\sum_{N_1}N_1^\frac14 \|A^{\pm_1}_{N_1}\|_{X^{0,b}_{\pm_1}})^2\|\phi_{\pm_2}\|_{X^{s+\frac12,b}_{\pm_2}}^2 \\
&\lesssim \sum_{N_0\ge1}N_0^{-2(\delta-\epsilon)}\|A_{\pm_1}\|_{X^{s,b}_{\pm_1}}^2\|\phi_{\pm_2}\|_{X^{s+\frac12,b}_{\pm_2}}^2.
\end{align*}
If $N_1\ll N_0\sim N_2$, then
\begin{align*}
\|Q_{jk}(D^{-1}A_{\pm_1},\phi_{\pm_2})\|_{X^{s-\frac12,b-1}_{\pm_0}}^2 &\lesssim \sum_{N_0}(N_0^\delta N_0^\epsilon \sum_{N_2}N_2^\frac34\|\phi^{\pm_2}_{N_2}\|_{X^{0,b}_{\pm_2}} )^2 \|A_{\pm_1}\|_{X^{s,b}_{\pm_1}}^2 \\
&\lesssim \sum_{N_0\ge1}N_0^{-2(\delta-\epsilon)}\|A_{\pm_1}\|_{X^{s,b}_{\pm_1}}^2\|\phi_{\pm_2}\|_{X^{s+\frac12,b}_{\pm_2}}^2.
\end{align*}
The case $N_2\ll N_0\sim N_1$ is similar. \\
Now we focus on $\mathbf J^1$. There are two important relation on frequency:
\begin{align}
&N_0\ll N_1\sim N_2,\quad \pm_1\neq\pm_2,\quad \theta_{12}\sim 1,\\
& N_{\min}^{12}\ll N_0\sim N_{\max}^{12},\quad \theta_{12}\ll1.
\end{align}
Since we have $\xi_0=\xi_1+\xi_2$, high-high-low interaction implies $\theta_{12}\sim 1$ and hence the null structure plays no crucial role. In this case we have high modulation with low frequency. Thus instead of \eqref{bi-selb-1} and \eqref{bi-selb-2}, we use \eqref{bi-selb-3}, a trivial volume estimate. It gives better estimate. In fact,
\begin{align*}
\mathbf J^1 &\lesssim \sum_{N_1,N_2}\sum_{L_1,L_2}N_0^{-\frac14+\delta}L_0^{-\frac12+\epsilon}N_2N_0(L_{\min}^{012})^\frac12\|A^{\pm_1}_{N_1,L_1}\|\|\phi^{\pm_2}_{N_2,L_2}\| \\
&\lesssim \sum_{N_1,N_2}\sum_{L_1,L_2}N_0^{\frac34+\delta}L_0^\epsilon N_2(N_{\min}^{12})^{-1}(L_1L_2)^\frac12\|A^{\pm_1}_{N_1,L_1}\|\|\phi^{\pm_2}_{N_2,L_2}\| \\
&\lesssim \sum_{N_1,N_2}\sum_{L_1,L_2}N_0^\delta L_0^\epsilon N_1^{\frac{3}{16}}N_2^{\frac{9}{16}}(L_1L_2)^\frac12\|A^{\pm_1}_{N_1,L_1}\|\|\phi^{\pm_2}_{N_2,L_2}\|.
\end{align*}
Hence we focus on the high-low-high interction $N_{\min}^{12}\ll N_0\sim N_{\max}^{12}$ with low modulation - high frequency.
\subsubsection{Case 1: $L_0\ll L_2$}
\begin{align*}
\mathbf J^1 &\lesssim \sum_{N_1,N_2}\sum_{L_1,L_2}N_0^{-\frac14+\delta}L_0^{-\frac12+\epsilon}N_2\left(\frac{L_2}{N_{\min}^{12}}\right)^\frac12 (N_{\min}^{012}L_{\min}^{01})^\frac12 (N_{\min}^{01}L_{\max}^{01})^\frac14\|A^{\pm_1}_{N_1,L_1}\|\|\phi^{\pm_2}_{N_2,L_2}\| \\
&\lesssim \sum_{N_1,N_2}\sum_{L_1,L_2}N_0^\delta L_0^\epsilon N_2N_0^{-\frac14} (N_{\min}^{01})^\frac14 L_0^{-\frac12}(L_{\min}^{01})^\frac12(L_{\max}^{01})^\frac14L_2^\frac12\|A^{\pm_1}_{N_1,L_1}\|\|\phi^{\pm_2}_{N_2,L_2}\| \\
&\lesssim  \sum_{N_1,N_2}\sum_{L_1,L_2}N_0^\delta N_1^\frac14 N_2^\frac34 L_0^\epsilon(L_1L_2)^\frac12\|A^{\pm_1}_{N_1,L_1}\|\|\phi^{\pm_2}_{N_2,L_2}\|.
\end{align*}
\subsubsection{Case 2: $L_2\ll L_0$}
\begin{align*}
\mathbf J^1 &\lesssim \sum_{N_1,N_2}\sum_{L_1,L_2}N_0^{-\frac14+\delta}L_0^{-\frac12+\epsilon}N_2\left(\frac{L_0}{N_{\min}^{12}}\right)^\frac12 (N_{\min}^{012}L_1)^\frac12 (N_{\min}^{12}L_2)^\frac14\|A^{\pm_1}_{N_1,L_1}\|\|\phi^{\pm_2}_{N_2,L_2}\| \\
&\lesssim \sum_{N_1,N_2}\sum_{L_1,L_2}N_0^\delta L_0^\epsilon N_2N_0^{-\frac14} (N_{\min}^{12})^\frac14 L_1^\frac12 L_2^\frac14\|A^{\pm_1}_{N_1,L_1}\|\|\phi^{\pm_2}_{N_2,L_2}\| \\
&\lesssim  \sum_{N_1,N_2}\sum_{L_1,L_2}N_0^\delta N_1^\frac14 N_2^\frac34 L_0^\epsilon(L_1L_2)^\frac12\|A^{\pm_1}_{N_1,L_1}\|\|\phi^{\pm_2}_{N_2,L_2}\|.
\end{align*}
This completes the proof of \eqref{bi-aphi-jk}.

\subsection{Proof of \eqref{bi-phiphi-jk}}
We write
\begin{align*}
\|Q_{jk}(\overline{\phi_{\pm_1}},\phi_{\pm_2})\|_{X^{s-1,b-1}_{\pm_0}}^2 &= \sum_{N_0,L_0}(N_0^{-\frac34+\delta}L_0^{-\frac12+\epsilon}\|P_{K_{N_0,L_0}^{\pm_0}}Q_{jk}(\overline{\phi_{\pm_1}},\phi_{\pm_2})\|)^2 \\
&= \sum_{N_0,L_0}(N_0^{-\frac34+\delta}L_0^{-\frac12+\epsilon}\sum_{N_1,N_2}\sum_{L_1,L_2}\|P_{K_{N_0,L_0}^{\pm_0}}Q_{jk}(\overline{\phi^{\pm_1}_{N_1,L_1}},\phi^{\pm_2}_{N_2,L_2})\|)^2 \\
&\lesssim \sum_{N_0,L_0}(N_0^{-\frac34+\delta}L_0^{-\frac12+\epsilon}\sum_{N_1,N_2}\sum_{L_1,L_2}N_1N_2\theta_{12}C_{N,L}^{012}\|\phi^{\pm_1}_{N_1,L_1}\|\|\phi^{\pm_2}_{N_2,L_2}\|)^2.
\end{align*}
We need to show the following:
\begin{align*}
\mathbf J^2 &:= \sum_{N_1,N_2}\sum_{L_1,L_2}N_0^{-\frac34+\delta}L_0^{-\frac12+\epsilon}N_1N_2\theta_{12}C_{N,L}^{012}\|\phi^{\pm_1}_{N_1,L_1}\|\|\phi^{\pm_2}_{N_2,L_2}\| \\
&\lesssim \sum_{N_1,N_2}\sum_{L_1,L_2}N_0^\delta L_0^\epsilon (N_1N_2)^\frac34 (L_1L_2)^\frac12\|\phi^{\pm_1}_{N_1,L_1}\|\|\phi^{\pm_2}_{N_2,L_2}\| .
\end{align*}
Note that we have two important interaction:
\begin{align}
& N_0\ll N_1\sim N_2,\quad \pm_1=\pm_2,\quad \theta_{12}\ll1,\\
& N_{\min}^{12}\ll N_0\sim N_{\max}^{12},\quad \theta_{12}\ll1.
\end{align}
\subsubsection{Case 1: $L_0\ll L_2$}
\begin{align*}
\mathbf J^2 &\lesssim \sum_{N_1,N_2}\sum_{L_1,L_2}N_0^{-\frac34+\delta}L_0^{-\frac12+\epsilon}N_1N_2\left(\frac{N_0L_2}{N_1N_2}\right)^\frac12 (N_{\min}^{012}L_{\min}^{01})^\frac12 (N_{\min}^{01}L_{\max}^{01})^\frac14\|\phi^{\pm_1}_{N_1,L_1}\|\|\phi^{\pm_2}_{N_2,L_2}\| \\
&\lesssim \sum_{N_1,N_2}\sum_{L_1,L_2}N_0^\delta L_0^\epsilon (N_{\min}^{012}N_1N_2)^\frac12 N_0^{-\frac14}(N_{\min}^{01})^\frac14 L_0^{-\frac12}(L_{\min}^{01})^\frac12 (L_{\max}^{01})^\frac14 L_2^\frac12\|\phi^{\pm_1}_{N_1,L_1}\|\|\phi^{\pm_2}_{N_2,L_2}\| \\
&\lesssim \sum_{N_1,N_2}\sum_{L_1,L_2}N_0^\delta L_0^\epsilon (N_1N_2)^\frac34 (L_1L_2)^\frac12 \|\phi^{\pm_1}_{N_1,L_1}\|\|\phi^{\pm_2}_{N_2,L_2}\|.
\end{align*}
\subsubsection{Case 2: $L_2\ll L_0$}
\begin{align*}
\mathbf J^2 &\lesssim \sum_{N_1,N_2}\sum_{L_1,L_2}N_0^{-\frac34+\delta}L_0^{-\frac12+\epsilon}N_1N_2\left(\frac{N_0L_0}{N_1N_2}\right)^\frac12 (N_{\min}^{012}L_1)^\frac12 (N_{\min}^{12}L_2)^\frac14\|\phi^{\pm_1}_{N_1,L_1}\|\|\phi^{\pm_2}_{N_2,L_2}\| \\
&\lesssim \sum_{N_1,N_2}\sum_{L_1,L_2}N_0^\delta L_0^\epsilon (N_{\min}^{012}N_1N_2)^\frac12 N_0^{-\frac14}(N_{\min}^{12})^\frac14 L_1^\frac12 L_2^\frac14\|\phi^{\pm_1}_{N_1,L_1}\|\|\phi^{\pm_2}_{N_2,L_2}\| \\
&\lesssim \sum_{N_1,N_2}\sum_{L_1,L_2}N_0^\delta L_0^\epsilon (N_1N_2)^\frac34 (L_1L_2)^\frac12\|\phi^{\pm_1}_{N_1,L_1}\|\|\phi^{\pm_2}_{N_2,L_2}\| .
\end{align*}

\section{Trilinear estimates}
In this section, we give the proof of the trilinear estimates in Section 4. Even though we cannot reveal null forms which we enjoy in the previous section, by the mercy of the sharp bilinear estimates of wave type \eqref{bi-selb-1}, \eqref{bi-selb-2}, we obtain the required estimates.
We shall show the following estimates.
\begin{align}
&\|A_{\pm_1}A_{\pm_2}\phi_{\pm_3}\|_{X^{s-\frac12,b-1}_\pm(S_T)} \lesssim \|A_{\pm_1}\|_{X^{s,b}_{\pm_1}}\|A_{\pm_2}\|_{X^{s,b}_{\pm_2}}\|\phi_{\pm_3}\|_{X^{s+\frac12,b}_{\pm_3}},\label{tri-aaphi} \\
& \|\overline{\phi_{\pm_1}}A_{\pm_2}\phi_{\pm_3}\|_{X^{s,b-1}_\pm(S_T)} \lesssim \|\phi_{\pm_1}\|_{X^{s+\frac12,b}_{\pm_1}}\|A_{\pm_2}\|_{X^{s,b}_{\pm_2}}\|\phi_{\pm_3}\|_{X^{s+\frac12,b}_{\pm_3}}\label{tri-phiaphi}.
\end{align}

The proof of \eqref{tri-aaphi} and \eqref{tri-phiaphi} is essentially same. We only prove the estimate \eqref{tri-aaphi}.



\subsection{Proof of \eqref{tri-aaphi}}
By the definition of $X^{s,b}_\pm$ we write
\begin{align*}
\|A_{\pm_1}A_{\pm_2}\phi_{\pm_3}\|_{X^{s-\frac12,b-1}_{\pm_4}}^2 &\lesssim \sum_{N_4,L_4}(N_4^{-\frac14+\delta}L_4^{-\frac12+\epsilon}\|P_{K_{N_4,L_4}^{\pm_4}}(A_{\pm_1}A_{\pm_2}\phi_{\pm_3})\|)^2 \\
& = \sum_{N_4,L_4}\left(N_4^{-\frac14+\delta}L_4^{-\frac12+\epsilon}\sup_{\|\varphi\|=1}\left|\int(A_{\pm_1}A_{\pm_2}\phi_{\pm_3})\overline{\varphi_{N_4,L_4}}\,dtdx\right|\right)^2\\
&\lesssim \sum_{N_4,L_4}\left(N_4^{-\frac14+\delta}L_4^{-\frac12+\epsilon}\sum_{N,L}\sup_{\|\varphi\|=1}\left|\int A^{\pm_1}_{N_1,L_1}A^{\pm_2}_{N_2,L_2}\phi^{\pm_3}_{N_3,L_3}\overline{\varphi^{\pm_4}_{N_4,L_4}}\,dtdx\right|\right)^2 \\
&\lesssim \sum_{N_4,L_4}\left(N_4^{-\frac14+\delta}L_4^{-\frac12+\epsilon}\sum_{N,L}C_{N,L}^{012}C_{N,L}^{034}\|A^{\pm_1}_{N_1,L_1}\|\|A^{\pm_2}_{N_2,L_2}\|\|\phi^{\pm_3}_{N_3,L_3}\|\right)^2
\end{align*}
To prove \eqref{tri-aaphi}, we need to show the following:
\begin{align*}
\mathbf K^1 &:= \sum_{N_1,N_2,N_3}\sum_{L_1,L_2,L_3}\sum_{N_0,L_0}N_4^{-\frac14+\delta}L_4^{-\frac12+\epsilon}C_{N,L}^{012}C_{N,L}^{034}\|A^{\pm_1}_{N_1,L_1}\|\|A^{\pm_2}_{N_2,L_2}\|\|\phi^{\pm_3}_{N_3,L_3}\| \\
&\lesssim \sum_{N_1,N_2,N_3}\sum_{L_1,L_2,L_3}N_4^\delta L_4^\epsilon (N_1N_2)^\frac14N_3^\frac34 (L_1L_2L_3)^\frac12\|A^{\pm_1}_{N_1,L_1}\|\|A^{\pm_2}_{N_2,L_2}\|\|\phi^{\pm_3}_{N_3,L_3}\|.
\end{align*}
Indeed, the general square summation by $L$ gives us
\begin{align*}
\|A_{\pm_1}A_{\pm_2}\phi_{\pm_3}\|_{X^{s-\frac12,b-1}_{\pm_4}}^2 &\lesssim \sum_{N_4\ge1}(N_4^\delta \sum_{N_1,N_2,N_3}(N_{\rm med}^{123})^\epsilon N_1^\frac14N_2^\frac14N_3^\frac34 \|A^{\pm_1}_{N_1}\|_{X^{0,b}_{\pm_1}}\|A^{\pm_2}_{N_2}\|_{X^{0,b}_{\pm_2}}\|\phi^{\pm_3}_{N_3}\|_{X^{0,b}_{\pm_3}})^2.
\end{align*}
If $N_4\sim N_3$, then
\begin{align*}
\|A_{\pm_1}A_{\pm_2}\phi_{\pm_3}\|_{X^{s-\frac12,b-1}_{\pm_4}}^2 &\lesssim \sum_{N_4\ge1}(N_4^\delta \sum_{N_1,N_2,N_3}(N_1N_2)^\epsilon N_1^\frac14N_2^\frac14N_3^\frac34 \|A^{\pm_1}_{N_1}\|_{X^{0,b}_{\pm_1}}\|A^{\pm_2}_{N_2}\|_{X^{0,b}_{\pm_2}}\|\phi^{\pm_3}_{N_3}\|_{X^{0,b}_{\pm_3}})^2 \\
&\lesssim \sum_{N_4}(N_4^\delta\sum_{N_3;N_3\sim N_4}N_3^\frac34\|\phi^{\pm_3}_{N_3}\|_{X^{0,b}_{\pm_3}})^2\|A_{\pm_1}\|_{X^{s,b}_{\pm_1}}^2\|A_{\pm_2}\|_{X^{s,b}_{\pm_2}}^2 \\
&\lesssim \|A_{\pm_1}\|_{X^{s,b}_{\pm_1}}^2\|A_{\pm_2}\|_{X^{s,b}_{\pm_2}}^2\|\phi_{\pm_3}\|_{X^{s+\frac12,b}_{\pm_3}}^2.
\end{align*}
If $N_4\gg N_3$, then we further treat three cases: $N_1\ll N_2$, $N_1\sim N_2$, and $N_1\gg N_2$. For $N_1\ll N_2$, we must have $N_4\sim N_2$ and then
\begin{align*}
\|A_{\pm_1}A_{\pm_2}\phi_{\pm_3}\|_{X^{s-\frac12,b-1}_{\pm_4}}^2 &\lesssim \sum_{N_4\ge1}(N_4^\delta \sum_{N_1,N_2,N_3}(N_1N_3)^\epsilon N_1^\frac14N_2^\frac14N_3^\frac34 \|A^{\pm_1}_{N_1}\|_{X^{0,b}_{\pm_1}}\|A^{\pm_2}_{N_2}\|_{X^{0,b}_{\pm_2}}\|\phi^{\pm_3}_{N_3}\|_{X^{0,b}_{\pm_3}})^2 \\
&\lesssim \sum_{N_4}(N_4^\delta\sum_{N_2;N_2\sim N_4}N_2^\frac14\|A^{\pm_2}_{N_2}\|_{X^{0,b}_{\pm_3}})^2 \|A_{\pm_1}\|_{X^{s,b}_{\pm_1}}^2\|\phi_{\pm_3}\|_{X^{s+\frac12,b}_{\pm_3}}^2 \\
&\lesssim \|A_{\pm_1}\|_{X^{s,b}_{\pm_1}}^2\|A_{\pm_2}\|_{X^{s,b}_{\pm_2}}^2\|\phi_{\pm_3}\|_{X^{s+\frac12,b}_{\pm_3}}^2.
\end{align*}
If $N_1\sim N_2$, then we have $N_4\lesssim N_2$ and hence
\begin{align*}
\|A_{\pm_1}A_{\pm_2}\phi_{\pm_3}\|_{X^{s-\frac12,b-1}_{\pm_4}}^2 &\lesssim \sum_{N_4\ge1}(N_4^\delta \sum_{N_1,N_2,N_3}(N_1N_3)^\epsilon N_1^\frac14N_2^\frac14N_3^\frac34 \|A^{\pm_1}_{N_1}\|_{X^{0,b}_{\pm_1}}\|A^{\pm_2}_{N_2}\|_{X^{0,b}_{\pm_2}}\|\phi^{\pm_3}_{N_3}\|_{X^{0,b}_{\pm_3}})^2 \\
&\lesssim \sum_{N_4}(N_4^\delta N_4^{\epsilon-2\delta})^2\|A_{\pm_1}\|_{X^{s,b}_{\pm_1}}^2\|A_{\pm_2}\|_{X^{s,b}_{\pm_2}}^2\|\phi_{\pm_3}\|_{X^{s+\frac12,b}_{\pm_3}}^2.
\end{align*}
The case $N_1\gg N_2$ is symmetric to $N_1\ll N_2$. Also $N_4\ll N_3$  can be treated similarly.\\
Now we deal with the estimate of $\mathbf K^1$. We assume
$$
L_1\le L_2,\quad L_3\le L_4.
$$

\subsubsection{Case 1: $L_0\ll L_2,L_4$}
\begin{align*}
\mathbf K^1 &\lesssim \sum_{N_1,N_2,N_3}\sum_{L_1,L_2,L_3}\sum_{N_0,L_0}N_4^{-\frac14+\delta}L_4^{-\frac12+\epsilon}(N_{\min}^{012}L_{\min}^{01}N_{\min}^{034}L_{\min}^{03})^\frac12(N_{\min}^{01}L_{\max}^{01}N_{\min}^{03}L_{\max}^{03})^\frac14\|A^{\pm_1}_{N_1,L_1}\|\|A^{\pm_2}_{N_2,L_2}\|\|\phi^{\pm_3}_{N_3,L_3}\| \\
&\lesssim \sum_{N_1,N_2,N_3}\sum_{L_1,L_2,L_3}N_4^\delta L_4^\epsilon (N_1N_2)^\frac14N_3^\frac34(L_1L_2L_3)^\frac12\|A^{\pm_1}_{N_1,L_1}\|\|A^{\pm_2}_{N_2,L_2}\|\|\phi^{\pm_3}_{N_3,L_3}\|.
\end{align*}

\subsubsection{Case 2: $L_4\ll L_0\ll L_2$}
\begin{align*}
\mathbf K^1 &\lesssim \sum_{N_1,N_2,N_3}\sum_{L_1,L_2,L_3}\sum_{N_0,L_0}N_4^{-\frac14+\delta}L_4^{-\frac12+\epsilon}(N_{\min}^{012}L_{\min}^{01}N_{\min}^{034}L_3)^\frac12(N_{\min}^{01}L_{\max}^{01}N_{\min}^{34}L_4)^\frac14\|A^{\pm_1}_{N_1,L_1}\|\|A^{\pm_2}_{N_2,L_2}\|\|\phi^{\pm_3}_{N_3,L_3}\| \\
&\lesssim \sum_{N_1,N_2,N_3}\sum_{L_1,L_2,L_3}N_4^\delta L_4^\epsilon (N_1N_2)^\frac14N_3^\frac34(L_1L_2L_3)^\frac12\|A^{\pm_1}_{N_1,L_1}\|\|A^{\pm_2}_{N_2,L_2}\|\|\phi^{\pm_3}_{N_3,L_3}\|.
\end{align*}

\subsubsection{Case 3: $L_2,L_4\ll L_0$}
\begin{align*}
\mathbf K^1 &\lesssim \sum_{N_1,N_2,N_3}\sum_{L_1,L_2,L_3}\sum_{N_0,L_0}N_4^{-\frac14+\delta}L_4^{-\frac12+\epsilon}(N_{\min}^{012}L_1N_{\min}^{034}L_3)^\frac12(N_{\min}^{12}L_2N_{\min}^{34}L_4)^\frac14\|A^{\pm_1}_{N_1,L_1}\|\|A^{\pm_2}_{N_2,L_2}\|\|\phi^{\pm_3}_{N_3,L_3}\| \\
&\lesssim \sum_{N_1,N_2,N_3}\sum_{L_1,L_2,L_3}N_4^\delta L_4^\epsilon (N_1N_2)^\frac14N_3^\frac34(L_1L_2L_3)^\frac12\|A^{\pm_1}_{N_1,L_1}\|\|A^{\pm_2}_{N_2,L_2}\|\|\phi^{\pm_3}_{N_3,L_3}\|.
\end{align*}

\subsubsection{Case 4: $L_2\ll L_0\ll L_4$}
\begin{align*}
\mathbf K^1 &\lesssim \sum_{N_1,N_2,N_3}\sum_{L_1,L_2,L_3}\sum_{N_0,L_0}N_4^{-\frac14+\delta}L_4^{-\frac12+\epsilon}(N_{\min}^{012}L_1N_{\min}^{034}L_{\min}^{03})^\frac12(N_{\min}^{12}L_2N_{\min}^{03}L_{\max}^{03})^\frac14\|A^{\pm_1}_{N_1,L_1}\|\|A^{\pm_2}_{N_2,L_2}\|\|\phi^{\pm_3}_{N_3,L_3}\| \\
&\lesssim \sum_{N_1,N_2,N_3}\sum_{L_1,L_2,L_3}N_4^\delta L_4^\epsilon (N_1N_2)^\frac14N_3^\frac34(L_1L_2L_3)^\frac12\|A^{\pm_1}_{N_1,L_1}\|\|A^{\pm_2}_{N_2,L_2}\|\|\phi^{\pm_3}_{N_3,L_3}\|.
\end{align*}

\section{Estimates of $\mathcal V(\phi,\phi^\dagger)$}
In this section we focus on the Higgs potential $\mathcal V(\phi,\phi^\dagger)$. We only need to prove the following estimates:
\begin{align}
\|\phi_{\pm_1}\phi_{\pm_2}\phi_{\pm_3}\|_{X^{s-\frac12,b-1}_{\pm}(S_T)} &\lesssim \prod_{j=1}^3\|\phi_{\pm_j}\|_{X^{s+\frac12,b}_{\pm_j}},\label{higgs-poten1} \\
\|\phi_{\pm_1}\phi_{\pm_2}\phi_{\pm_3}\phi_{\pm_4}\phi_{\pm_5}\|_{X^{s-\frac12,b-1}_{\pm}(S_T)} &\lesssim \prod_{j=1}^5\|\phi_{\pm_j}\|_{X^{s+\frac12,b}_{\pm_j}}.\label{higgs-poten2}
\end{align}

The proof of \eqref{higgs-poten1} is already treated in Section 6. We focus on the proof of \eqref{higgs-poten2}.

\subsection{Proof of \eqref{higgs-poten2}}
The quintic term seems very cumbersome, however, we can prove the estimate easily by using Bernstein's inequality \eqref{Bernstein-ineq}. Indeed,
\begin{align*}
\|\phi_{\pm_1}\phi_{\pm_2}\phi_{\pm_3}\phi_{\pm_4}\phi_{\pm_0}\|_{X^{s-\frac12,b-1}_{\pm_0}} &\lesssim \sum_{N_0,L_0}(N_0^{-\frac14+\delta}L_0^{-\frac12+\epsilon}\|P_{K_{N_0,L_0}^{\pm_0}}(\phi_{\pm_1}\phi_{\pm_2}\phi_{\pm_3}\phi_{\pm_4}\phi_{\pm_0})\|)^2 \\
&\lesssim \sum_{N_0,L_0}(N_0^{-\frac14+\delta}L_0^{-\frac12+\epsilon}\|P_{K_{N_0,L_0}^{\pm_0}}(\phi_{\pm_1}\phi_{\pm_2}\phi_{\pm_3}\phi_{\pm_4})\|_{L^4}\|P_{K_{N_0,L_0}^{\pm_0}}\phi_{\pm_0}\|_{L^4})^2 \\
&\lesssim \sum_{N_0,L_0}(N_0^{\frac34+\delta}L_0^{\epsilon}\|P_{K_{N_0,L_0}^{\pm_0}}(\phi_{\pm_1}\phi_{\pm_2}\phi_{\pm_3}\phi_{\pm_4})\|\|P_{K_{N_0,L_0}^{\pm_0}}\phi_{\pm_0}\|)^2 \\
&\lesssim \sum_{N_0,L_0}(N_0^{\frac34+\delta}L_0^{\frac12+\epsilon}N_0\|P_{K_{N_0,L_0}^{\pm_0}}(\phi_{\pm_1}\phi_{\pm_2})\|\|P_{K_{N_0,L_0}^{\pm_0}}(\phi_{\pm_3}\phi_{\pm_4})\|\|\phi^{\pm_0}_{N_0,L_0}\|)^2 \\
&\lesssim \sum_{N_0,L_0}(N_0^{\frac34+\delta}L_0^{\frac12+\epsilon}N_0\sum_{N,L}\|P_{K_{N_0,L_0}^{\pm_0}}(\phi^{\pm_1}_{N_1,L_1}\phi^{\pm_2}_{N_2,L_2})\|\|P_{K_{N_0,L_0}^{\pm_0}}(\phi^{\pm_3}_{N_3,L_3}\phi^{\pm_4}_{N_4,L_4})\|\|\phi^{\pm_0}_{N_0,L_0}\|)^2
\end{align*}
Thus we need to show that
\begin{align*}
\mathbf K &:= N_0C_{N,L}^{012}C_{N,L}^{034}\|\phi^{\pm_1}_{N_1,L_1}\|\|\phi^{\pm_2}_{N_2,L_2}\|\|\phi^{\pm_3}_{N_3,L_3}\|\phi^{\pm_4}_{N_4,L_4}\|\|\phi^{\pm_0}_{N_0,L_0}\| \\
&\lesssim (N_1N_2N_3N_4)^\frac34 (L_1L_2L_3L_4)^\frac12\|\phi^{\pm_1}_{N_1,L_1}\|\|\phi^{\pm_2}_{N_2,L_2}\|\|\phi^{\pm_3}_{N_3,L_3}\|\phi^{\pm_4}_{N_4,L_4}\|\|\phi^{\pm_0}_{N_0,L_0}\|.
\end{align*}
This is trivial, since $C_{N,L}^{012}C_{N,L}^{034}$ contains $N^\frac32$ and $N^1N^\frac32 = N^\frac52\ll N^{\frac{12}{4}}$.\\
This completes the proof of Theorem \ref{lwp}.


\section{The failure of smoothness}
In this section we show the flow $(\phi(0), A_\mu(0)) \mapsto (\phi(t), A_\mu(t))$ is not $C^3$ near the origin. Especially, we show the gauge field is not smooth at the origin in $H^\frac14$.
For this purpose we consider the system given by
\begin{align}\label{csh-del}
\left\{
\begin{array}{l}
\square\phi = -2\left[A^{\mu}, \partial_{\mu} \phi\right]-\left[A_{\mu},\left[A^{\mu}, \phi\right]\right] + \mathcal V(\phi,\phi^\dagger), \\
\square A_{\mu} = \left[\partial^{\nu} A_{\mu}, A_{\nu}\right] - \epsilon_{\mu \nu \alpha}\left(Q^{\nu \alpha}\left(\phi^{\dagger}, \phi\right)+Q^{\nu \alpha}\left(\phi, \phi^{\dagger}\right)\right) -\epsilon_{\mu \nu \alpha}\partial^{\nu}\left(\left[\phi^{\dagger},\left[A^{\alpha}, \phi\right]\right]-\big[\left[A^{\alpha}, \phi\right]^{\dagger}, \phi\big]\right) \\
 (\phi, \partial_t \phi)(0) = (\delta f, 0), \quad A_0(0) = \delta a_0, \quad A_j(0) = 0 \quad \partial_t A_0(0) = 0,\\
 \partial_t A_j(0) = \delta \partial_j a_0 + \delta^2\epsilon_{0jk}([f^\dagger, \partial^k f] - [(\partial^k f)^\dagger, f]),
\end{array}\right.\end{align}
where $0<\delta\ll1$ and $f$ is a $\frak g$-valued smooth function. We denote the local solution of \eqref{csh-del} by $(\phi(\delta,t), A_\mu(\delta,t))$. If $f = f_1T^1$ for a smooth scalar function $f_1$, then $\partial_t A_j(0) = \delta \partial_j a_0 $.

\subsection{Set up}
We prove by contradiction. Assume that the flow is $C^3$ at the origin in $H^s$ and $f = f_1 T^1$ and $a_0 = a_{0, 2}T^2$. Since $(\phi, \partial_t \phi)(\delta = 0, t = 0) = (0, 0)$ and $(A_\mu, \partial_tA_\mu)(\delta = 0, t = 0) = (0, 0)$, the solution $\phi(\delta = 0, t) = 0$ and $A_\mu(\delta = 0, t) = 0$. By taking derivative to \eqref{csh-del} w.r.t. $\delta$ we see that since
$$
A_\mu(\delta = 0, 0) = 0,\quad \partial_\delta A_\mu(\delta = 0, t = 0) =  \delta_{\mu 0}a_0,\quad \partial_t\partial_\delta A_\mu(\delta = 0, t = 0) = \delta_{\mu j}\partial_j a_0,
$$
the solution $\partial_\delta A_\mu(\delta = 0, t)$ is written as
$$
\partial_\delta A_\mu(\delta = 0, t) = \sum_{\pm} A_{\mu, \pm}^{\rm hom} = \sum_{\pm} \frac12e^{\mp it D}\left(\delta_{\mu 0}a_{0,2} \pm \frac1{iD}\delta_{\mu j}\partial^j a_{0,2}\right)T^2.
$$
On the other hand, $\partial_\delta \phi$ satisfies from the formula \eqref{phi-a} below that
$$
\square \partial_\delta\phi(\delta = 0) = -m^2\partial_\delta \phi(\delta = 0),\quad (\partial_\delta \phi(\delta = 0, t = 0), \partial_t\partial_\delta \phi(\delta = 0, t = 0)) = (f, 0),
$$
where $m^2 = 2v^4$. Then the solution $\partial_\delta \phi_a(\delta = 0) = 0$ for $a > 1$, and hence
$$
\partial_\delta \phi(\delta = 0) = \sum_{\pm}\frac12 e^{\mp it\sqrt{m^2-\Delta}}f_1T^1.
$$

Let us consider the second derivatives.
\begin{align*}
\partial_\delta^2 \left[\partial^{\nu} A_{\mu}, A_{\nu}\right](\delta = 0, t) = 2\sum_{\pm_1, \pm_2}\left[\partial^{\nu} A_{\mu, \pm_1 }^{\rm hom}(t), A_{\nu, \pm_2}^{\rm hom}(t)\right] = 2\sum_{\pm_1, \pm_2}\left(\cdots \right)[T^2, T^2] = 0.
\end{align*}
\begin{align*}
\partial_\delta^2(Q^{\nu \alpha}\left(\phi^{\dagger}, \phi\right) &+Q^{\nu \alpha}\left(\phi, \phi^{\dagger}\right))(\delta = 0, t)\\
& = \frac12 \sum_{\pm_1, \pm_2}Q^{\nu \alpha}\left((e^{\mp_1 it\sqrt{m^2-\Delta}}f_1)^*T^1, e^{\mp_2 it\sqrt{m^2-\Delta}}f_1T^1\right)\\
&\qquad  + \frac12 \sum_{\pm_1, \pm_2} Q^{\nu \alpha}\left(e^{\mp_1 it\sqrt{m^2-\Delta}}f_1 T^1, (e^{\mp_2 it\sqrt{m^2-\Delta}}f_1)^* T^1\right) \\
&= 0 \qquad (\because Q^{\nu\alpha}(h_1T^1, g_1T^1) = -Q^{\nu\alpha}(g_1T^1, h_1T^1)).
\end{align*}
Clearly the second derivative of cubic term of RHS of $\square A_\mu$ is $0$ at $\delta = 0$. Hence $\partial_\delta^2A_\mu(\delta = 0, t) = 0$.

Taking the second derivative to the equation of $\phi$ and then removing the cubic terms, since the initial data are zero and the linear term plays a role of mass, we have
\begin{align}\label{2nd-der}
\partial_\delta^2\phi(\delta=0,t) = -\sum_{\pm_1,\pm_2,\pm_3}\int_0^te^{\mp_1i(t-t')\sqrt{m^2-\Delta}}(2iD)^{-1}[A^{\rm hom}_{\mu,\pm_2}(t'),\partial^\mu e^{\mp_3it'\sqrt{m^2-\Delta}}f]\,dt',
\end{align}
If the flow is $C^2$ in $H^{s} \times H^\sigma$, within a local existence time interval, the following inequality holds:
\begin{align}\label{c2-csh-ineq}
\sup_{0\le t\le T}\|\partial_\delta^2\phi(\delta = 0, t)\|_{H^s} \lesssim \|f_1\|_{H^s}\|a_0\|_{H^{\sigma}}.
\end{align}

In order to treat the third derivative, let us first observe that
\begin{align*}
&Q^{\nu\alpha}(\partial_\delta^2\phi^\dagger(\delta = 0), \partial_\delta \phi(\delta = 0)) + Q^{\nu\alpha}(\partial_\delta \phi(\delta = 0), \partial_\delta^2\phi^\dagger(\delta = 0))\\
&= \left(\cdots\right)([T^2, T^1]T^1 + [T^1, [T^2, T^1]] = \left(\cdots\right)([T^2, T^1]T^1 + T^1 [T^2, T^1]) =  \left(\cdots\right)[T^2, T^1T^1].
\end{align*}
It is an easy exercise to show that $[T^a, T^bT^b] = 0$ for any generator $T^a$ and $T^b$. (For instance see the page 420 of \cite{sred}.)
Therefore, we obtain
\begin{align}\begin{aligned}\label{3rd-d-a}
&\partial_\delta^3A_2(\delta = 0) \\
&= \frac {3i}2\sum_{\pm_1, \pm_2, \pm_3, \pm_4}\mp_1 \int_0^t\frac{e^{\mp_1 i(t-t')D}}{2iD}\epsilon_{201}\partial^{0}\left(\left[(e^{\mp_2it'\sqrt{m^2-\Delta}}f)^{\dagger},\left[A_{\pm_3}^{{\rm hom}, 1}, e^{\mp_4it'\sqrt{m^2-\Delta}}f\right]\right] \right.\\
&\qquad\qquad\qquad\qquad\qquad\qquad\qquad\qquad\qquad \left. - \left[\left[A_{\pm_2}^{{\rm hom}, 1}, e^{\mp_3it'\sqrt{m^2 - \Delta}}f\right]^{\dagger}, e^{\mp_4it'\sqrt{m^2-\Delta}}f\right]\right)dt'\\
&\quad+ \frac {3i}2\sum_{\pm_1, \pm_2, \pm_3, \pm_4}\mp_1 \int_0^t\frac{e^{\mp_1 i(t-t')D}}{2iD}\epsilon_{210}\partial^{1}\left(\left[(e^{\mp_2it'\sqrt{m^2-\Delta}}f)^{\dagger},\left[A_{\pm_3}^{{\rm hom}, 0}, e^{\mp_4it'\sqrt{m^2-\Delta}}f\right]\right] \right.\\
&\qquad\qquad\qquad\qquad\qquad\qquad\qquad\qquad\qquad \left. - \left[\left[A_{\pm_2}^{{\rm hom}, 0}, e^{\mp_3it'\sqrt{m^2 - \Delta}}f\right]^{\dagger}, e^{\mp_4it'\sqrt{m^2-\Delta}}f\right]\right)dt'\\
&=: I + I\!I.
\end{aligned}\end{align}
If the flow is $C^3$ in $H^{s} \times H^\sigma$ within a local existence time interval, then we have
\begin{align}\label{csh-ineq}
\sup_{0\le t\le T}\|\partial_\delta^3A_0(\delta = 0, t)\|_{H^\sigma} \lesssim \|f_1\|_{H^s}^2\|a_0\|_{H^{\sigma}}.
\end{align}

\subsection{Failure of \eqref{csh-ineq} when $s < \frac12$}
Given $\lambda \gg 1 + m^2$, let us define $W_\lambda$ by $\{\xi = (\xi_1, \xi_2) : |\xi_1-\lambda| \le 10^{-6}\lambda,\;\; |\xi_2| \le 10^{-6}\lambda^\frac12\}$, and $-W_\lambda = \{\xi : -\xi \in W_\lambda\}$.
Then for any $\xi \in 2W_{\lambda}$ there exists $0 < \rho < 1$ such that $2(1-\rho)\lambda \le |\xi| \le 2(1 + \rho)\lambda$.

We now take $f_1$ and $a_{0, 2}$ as
$$
\widehat f_1(\xi) = \chi_{W_\lambda},\quad \widehat a_{0, 2} = \chi_{2W_\lambda}(\xi).
$$

Using integration by parts and the fact that $A_{\pm_3}^{{\rm hom, 1}} = \pm_3e^{\mp_3it'D}\frac1{2iD}\partial_1a_{0, 2}T^2 $, we have
\begin{align*}
I &= \frac{3}{4}\sum_{\pm_1, \cdots, \pm_4} \mp_1\frac1{D}\left([(e^{\mp_2it\sqrt{m^2-\Delta}}f)^\dagger, [A_{\pm_3}^{{\rm hom}, 1}(t), e^{\mp_4it\sqrt{m^2-\Delta}}f]]\right.\\
&\qquad\qquad\qquad\qquad\qquad\qquad\qquad\left.+ [[e^{\mp_2it\sqrt{m^2-\Delta}}f, A_{\pm_3}^{{\rm hom, 1}}(t)]^\dagger, e^{\mp_4it\sqrt{m^2-\Delta}}f]\right)\\
&\qquad\qquad + \frac34\sum_{\pm_1, \cdots, \pm_4}\!\pm_1\int_0^t e^{\mp_1i(t-t')D}\left([(e^{\mp_2it'\sqrt{m^2-\Delta}}f)^\dagger, [A_{\pm_3}^{{\rm hom}, 1}(t'), e^{\mp_4it'\sqrt{m^2-\Delta}}f]]\right.\\
&\qquad\qquad\qquad\qquad\qquad\qquad\qquad \left.+[[e^{\mp_2it'\sqrt{m^2-\Delta}}f, A_{\pm_3}^{{\rm hom, 1}}(t')]^\dagger, e^{\mp_4it\sqrt{m^2-\Delta}}f]\right)dt'\\
&=: I_1 + I_2.
\end{align*}

Taking space Fourier transform $\mathcal F_x$ for $\xi \in 2W_{\lambda}$ gives
$$
|\mathcal F_x I_1(t, \xi)| \lesssim \frac1{|\xi|}\int\!\!\int_{\mathbb R^2 \times \mathbb R^2} \chi_{W_\lambda}(\zeta-\eta)\chi_{W_\lambda}(\zeta)\chi_{ W_\lambda}(\xi-\eta)\,d\zeta d\eta \lesssim \lambda^2.
$$
Then the direct integration w.r.t. $t'$ gives us
\begin{align*}
\mathcal F_x I_2(t, \xi) &= \frac{3i}{8}\sum_{\pm_1, \cdots, \pm_4} \pm_1 e^{\mp_1it|\xi|}\int\!\!\int_{\mathbb R^2 \times \mathbb R^2}\mathbf m_{1234} \chi_{ W_\lambda}(\zeta-\eta)\left(\pm_3\frac{\zeta_1}{|\zeta|}\right)\chi_{2W_\lambda}(\zeta)\chi_{  W_\lambda}(\xi-\eta)\,d\zeta d\eta [T^1, [T^2,T^1]]\\
&\quad + \frac{3i}{8}\sum_{\pm_1, \cdots, \pm_4} \pm_1 e^{\mp_1it|\xi|}\int\!\!\int_{\mathbb R^2 \times \mathbb R^2}\widetilde{\mathbf m}_{1234} \chi_{W_\lambda}(\zeta-\eta)\left(\pm_3\frac{\zeta_1}{|\zeta|}\right)\chi_{2W_\lambda}(-\zeta)\chi_{W_\lambda}(\xi-\eta)\,d\zeta d\eta[T^1, [T^2,T^1]]\\
&=: (\mathcal N_{1234} + \widetilde{\mathcal N}_{1234})[T^1, [T^2,T^1]],
\end{align*}
where
\begin{align*}
&\mathbf m_{1234}(t,\xi,\eta,\zeta) = \frac{e^{it\omega_{1234}}-1}{i\omega_{1234}},\quad \widetilde{\mathbf m}_{1234}(t,\xi,\eta,\zeta) = \frac{e^{it\widetilde{\omega}_{1234}}-1}{i\widetilde{\omega}_{1234}},\\
&\omega_{1234} = \pm_1|\xi| \pm_2\sqrt{m^2+|\zeta - \eta|^2} -\pm_3|\zeta| -\pm_4\sqrt{m^2+|\xi-\eta|^2},\\
&\widetilde{\omega}_{1234} = \pm_1|\xi| \pm_2\sqrt{m^2+|\zeta - \eta|^2} + \pm_3|\zeta| -\pm_4\sqrt{m^2+|\xi-\eta|^2}.
\end{align*}
For $\xi \in 2W_{\lambda}$ from the support conditions of the integrand of $\mathcal N_{1234}$ that $\eta$ is at least in $W_\lambda$. On the other hand, the support condition of integrand in $\widetilde{\mathcal N}_{1234}$ enforces that
$$
\eta \in (- W_\lambda - 2W_\lambda) \cap (2W_{\lambda} -  W_\lambda),
$$
which is impossible. Hence $\widetilde{\mathcal N}_{1234}$ is vanishing.

Now by Taylor expansion $\omega_{1234}$ can be estimated as follows: if $(\pm_1 = \pm_3) \& (\pm_2 = \pm_4)$, or $(\pm_1 = \pm_3) \neq (\pm_2 = \pm_4)$, or $(\pm_1, \pm_3) = (+,-) \& (\pm_2, \pm_4) = (-, +)$, or $(-, +) \& (+, -)$ (these cases are called {\it resonance} ones denoted by $\mathcal R$), then
\begin{align}\label{res}
|\omega_{1234}| \lesssim \lambda^\frac12.
\end{align}
Otherwise,
\begin{align}\label{res-comp}
|\omega_{1234}| \sim \lambda.
\end{align}

If $\pm_1, \cdots, \pm_4 \in \mathcal R$, then by taking $t = \varepsilon \lambda^{-\frac12}$ for some fixed $\lambda^{-\frac12} \ll \varepsilon \ll 1$, from \eqref{res} we obtain
$$
\mathbf m_{1234} = t(1 + O_{1234}(\varepsilon)),
$$
where $O_{1234}(\varepsilon)$ is $O(\varepsilon)$ depending on $\pm_1, \cdots, \pm_4$. If $\pm_1, \cdots, \pm_4 \in \mathcal R^c$, then by \eqref{res-comp}
$$
|\mathbf m_{1234}| \lesssim \lambda^{-1}.
$$
Using these estimates we have
$$
\left|\mathcal N_{1234}\right| \ge \left|\sum_{\pm_1,\cdots, \pm_4 \in \mathcal R}\left(\cdots\right)\right| - \left|\sum_{\pm_1,\cdots, \pm_4 \in \mathcal R^c}\left(\cdots\right)\right|
$$
In the summation in $\mathcal R$ there are four cases of $\pm_1 = \pm_3$ and two cases of $\pm_1 \neq \pm_3$, which results in
\begin{align*}
\left|\sum_{\pm_1,\cdots, \pm_4 \in \mathcal R}\left(\cdots\right)\right| &\ge \frac34 t\left| \int\!\!\int_{\mathbb R^2\times \mathbb R^2}\cos{t|\xi|} \frac{\zeta_1}{|\zeta|}\chi_{ W_\lambda}(\zeta-\eta)\chi_{2W_\lambda}(\zeta)\chi_{W_\lambda}(\xi-\eta)\,d\zeta d\eta\right| \\
&\qquad - Ct\varepsilon\int\!\!\int_{\mathbb R^2\times \mathbb R^2} \frac{\zeta_1}{|\zeta|}\chi_{W_\lambda}(\zeta-\eta)\chi_{2W_\lambda}(\zeta)\chi_{W_\lambda}(\xi-\eta)\,d\zeta d\eta\\
&\ge \frac34t|\cos(t|\xi|)| \int\!\!\int_{\mathbb R^2\times \mathbb R^2} \frac{\zeta_1}{|\zeta|}\chi_{W_\lambda}(\zeta-\eta)\chi_{2W_\lambda}(\zeta)\chi_{ W_\lambda}(\xi-\eta)\,d\zeta d\eta \\
&\qquad- Ct\varepsilon\int\!\!\int_{\mathbb R^2\times \mathbb R^2}\frac{\zeta_1}{|\zeta|}\chi_{W_\lambda}(\zeta-\eta)\chi_{2W_\lambda}(\zeta)\chi_{W_\lambda}(\xi-\eta)\,d\zeta d\eta.
\end{align*}
Let us take $\lambda = \lambda(\varepsilon, \rho, k)$ for any large integer $k$ such that
$$
\frac{2k\pi-\varepsilon}{\varepsilon(1-\rho)} \le \lambda^\frac12 \le \frac{2k\pi+\varepsilon}{\varepsilon(1+\rho)}.
$$
Then
$$
\left|\sum_{\pm_1,\cdots, \pm_4 \in \mathcal R}\left(\cdots\right)\right| \gtrsim t(1 - C\varepsilon) \int\!\!\int_{\mathbb R^2\times \mathbb R^2}\chi_{W_\lambda}(\zeta)\chi_{W_\lambda}(\eta) \,d\zeta d\eta \gtrsim t\lambda^3 = \varepsilon \lambda^{\frac52}
$$
Since $\varepsilon \gg \lambda^{-\frac12}$, we have
\begin{align}\label{I-bound}
|\mathcal F_x(I)(t, \xi)| \gtrsim \varepsilon \lambda^{\frac52}.
\end{align}

Now let us consider $I\!I$. By the same argument as above we have
\begin{align*}
\mathcal F_x(I\!I)(t, \xi) &= \frac{3i}8 \sum_{\pm_1, \cdots, \pm_4} \mp_1 e^{\mp_1it|\xi|}\frac{\xi_1}{|\xi|} \int\!\!\int_{\mathbb R^2\times \mathbb R^2} \mathbf m_{1234} \chi_{W_\lambda}(\zeta-\eta)\chi_{2W_\lambda}(\zeta)\chi_{W_\lambda}(\xi-\eta)\,d\zeta d\eta\\
&= \frac{3i}8 \left(\sum_{\mathcal R}(\cdots) + \sum_{\mathcal R^c}(\cdots)\right).
\end{align*}
The estimate \eqref{res-comp} yields
$$
\left|\frac{3i}8 \sum_{\mathcal R^c}(\cdots)\right| \lesssim \lambda^2.
$$
On the other hand, \eqref{res} leads us to
\begin{align*}
\left|\sum_{\mathcal R}\left(\cdots\right)\right| \lesssim  t(|\sin (t|\xi|) + \varepsilon ) \int\!\!\int_{\mathbb R^2\times \mathbb R^2} \chi_{W_\lambda}(\zeta-\eta)\chi_{2W_\lambda}(\zeta)\chi_{W_\lambda}(\xi-\eta)\,d\zeta d\eta.
\end{align*}
By the choice of $\lambda$, $|\sin(t|\xi|)| \lesssim \varepsilon$. This implies that
\begin{align*}
\left|\sum_{\mathcal R}\left(\cdots\right)\right| \lesssim  \varepsilon\lambda^{-\frac12}\varepsilon \lambda^3 \lesssim \varepsilon^2\lambda^\frac52.
\end{align*}

Combining this with \eqref{I-bound} we conclude that for $\xi \in 2W_\lambda$
$$
|\mathcal F_x(\partial_\delta^3 A_{2}(\delta = 0))(t, \xi)| \gtrsim \varepsilon \lambda^\frac52.
$$
Now let us invoke the necessary-sufficient condition of \eqref{csh-ineq} for the $C^3$ smoothness, from which
we deduce that
$$
\varepsilon \lambda^\frac52 \lambda^{\sigma + \frac34} \lesssim \lambda^{2s + \frac32}\lambda^{\sigma + \frac34}
$$
and hence regardless of $\sigma$, $s$ should be greater than or equal to $\frac12$ for the $C^3$ smoothness. This completes the proof of Theorem \ref{fail-smooth}.

\subsection{Failure of \eqref{c2-csh-ineq} when $\sigma < \frac14$}
Let us set $f = f_1T^1$, $a_0 = a_{0, 2}T^2$ and $\widehat{f_1} = \widehat{a_{0, 2}} = \chi_{ W_\lambda}$.
By taking Fourier transform to \eqref{2nd-der} we have
\begin{align*}
&\mathcal F_x(\partial_\delta^2\phi(\delta = 0)(\xi)\\
& = -\frac14 \sum_{\pm_1, \pm_2, \pm_3} \frac{e^{\mp_1it\sqrt{m^2+|\xi|^2}}}{|\xi|}\int \widetilde{
\widetilde{\mathbf m}}_{123} \chi_{W_\lambda}(\xi-\eta)\left(\mp_3 \sqrt{m^2 + |\eta|^2} \pm_2\frac{\xi_1-\eta_1}{|\xi-\eta|}\right)\chi_{W_\lambda}(\eta)\,d\eta \;[T^2, T^1],
\end{align*}
where
$$
\widetilde{\widetilde{\mathbf m}}_{123} = \frac{e^{it\omega_{123}} - 1}{i\omega_{123}},\quad \omega_{123} = \pm_1\sqrt{m^2+|\xi|^2} - \pm_2|\xi-\eta| - \pm_3|\eta|.
$$
The resonance case $\mathcal R$ occurs only when $\pm_1 = \pm_2 = \pm_3$ for which $|\omega_{123}| \lesssim \lambda^\frac12$. For other cases $|\omega_{123}| \sim \lambda$.

If we take $t = \varepsilon \lambda^{-\frac12}$ and $\xi \in 2 W_\lambda$, then
\begin{align*}
&\mathcal F_x(\partial_\delta^2\phi(\delta = 0)(\xi)\\
& = -\frac14 \sum_{\pm_1 = \pm_2 = \pm_3} \frac{e^{\mp_1it\sqrt{m^2+|\xi|^2}}}{|\xi|}\int t(1 + O_{123}(\varepsilon)) \chi_{{ W_\lambda}}(\xi-\eta)\left(\mp_3 \sqrt{m^2 + |\eta|^2} \pm_2\frac{\xi_1-\eta_1}{|\xi-\eta|}\right)\chi_{{W_\lambda}}(\eta)\,d\eta\;[T^2, T^1]\\
&\qquad-\frac14 \sum_{\mathcal R^c} \frac{e^{\mp_1it\sqrt{m^2+|\xi|^2}}}{|\xi|}\int \widetilde{\widetilde{\mathbf m}}_{123} \chi_{{ W_\lambda}}(\xi-\eta)\left(\mp_3 \sqrt{m^2 + |\eta|^2} \pm_2\frac{\xi_1-\eta_1}{|\xi-\eta|}\right)\chi_{{ W_\lambda}}(\eta)\,d\eta\;[T^2, T^1]\\
&= -\frac t4 \sum_{\pm_1 = \pm_2 = \pm_3}\mp_1 \frac{e^{\mp_1it\sqrt{m^2+|\xi|^2}}}{|\xi|}\int \chi_{{ W_\lambda}}(\xi-\eta)\left(\sqrt{m^2 + |\eta|^2}\right)\chi_{{ W_\lambda}}(\eta)\,d\eta\;[T^2 ,T^1]\\
&\qquad +\frac t4 \sum_{\pm_1 = \pm_2 = \pm_3} \mp_1\frac{e^{\mp_1it\sqrt{m^2+|\xi|^2}}}{|\xi|}\int \chi_{{ W_\lambda}}(\xi-\eta)\left( \frac{\xi_1-\eta_1}{|\xi-\eta|}\right)\chi_{{ W_\lambda}}(\eta)\,d\eta\;[T^2, T^1]\\
&\qquad - \frac t4 \sum_{\pm_1 = \pm_2 = \pm_3}\mp_1 \frac{e^{\mp_1it\sqrt{m^2+|\xi|^2}}}{|\xi|}\int O_{123}(\varepsilon) \chi_{{W_\lambda}}(\xi-\eta)\left(\sqrt{m^2 + |\eta|^2} - \frac{\xi_1-\eta_1}{|\xi-\eta|}\right)\chi_{{ W_\lambda}}(\eta)\,d\eta\;[T^2 ,T^1]\\
&\qquad-\frac14 \sum_{\mathcal R^c} \frac{e^{\mp_1it\sqrt{m^2+|\xi|^2}}}{|\xi|}\int \widetilde{\widetilde{\mathbf m}}_{123} \chi_{{ W_\lambda}}(\xi-\eta)\left(\mp_3 \sqrt{m^2 + |\eta|^2} \pm_2\frac{\xi_1-\eta_1}{|\xi-\eta|}\right)\chi_{{ W_\lambda}}(\eta)\,d\eta\;[T^2, T^1]\\
&=: I + I\!I + I\!I\!I + I\!V.
\end{align*}
A direct calculation gives
$$
|I\!I| \lesssim t\lambda^\frac12 = \varepsilon, \quad |I\!I\!I| \lesssim \varepsilon^2 \lambda,
$$
and for the non-resonance case $\mathcal R^c$, $|\widetilde{\widetilde{\mathbf m}}_{123}|\lesssim\lambda^{-1}$ holds and hence $|I\!V| \lesssim \lambda^\frac12$.

As for $I$ we further take $\lambda = \lambda(\varepsilon, \rho, k)$ such that $$ \frac{k\pi + \frac\pi8}{\varepsilon(1-\rho)} \le \lambda^\frac12 \le  \frac{k\pi + \frac{3\pi}8}{\varepsilon(1+\rho)}$$ for each integer $k$. Then for $\xi \in 2 W_\lambda$
\begin{align*}
|I| &= \frac t2 \frac{|\sin{t|\xi|}|}{|\xi|}\int \chi_{{ W_\lambda}}(\xi-\eta)\left(\sqrt{m^2 + |\eta|^2}\right)\chi_{{ W_\lambda}}(\eta)\,d\eta\;\|[T^2, T^1]\|\\
&= \frac t2 \frac{|\sin{t|\xi|}|}{|\xi|}\int \chi_{{ W_\lambda}}(\xi-\eta)\left(\sqrt{m^2 + |\eta|^2}\right)\chi_{{ W_\lambda}}(\eta)\,d\eta\;\|[T^2, T^1]\|\\
&\gtrsim t |\sin{t|\xi|}|\lambda^\frac32\\
& \gtrsim \varepsilon \lambda.
\end{align*}
Hence we get
$$
|\mathcal F_x(\partial_\delta^2\phi(\delta = 0)(\xi)| \gtrsim \varepsilon \lambda.
$$
Suppose that \eqref{c2-csh-ineq} holds. Then
$$
\varepsilon \lambda \lambda^{s+\frac34} \lesssim \lambda^{s+\frac34}\lambda^{\sigma+\frac34}.
$$
Therefore, $\sigma$ should be greater than equal to $\frac14$.

\section{Appendix}

We introduce a set of infinitesimal generators $T^a$, $a = 1, 2, \cdots, n^2-1$, of Lie algebra $\mathfrak g$ which are traceless Hermitian matrices. These matrices obey the normalization condition ${\rm Tr}(T^a T^b) = 2\delta^{ab}$, where $\delta^{ab}$ is Kronecker delta.
 Now we introduce so-called structure coefficients $f^{ab}_{\phantom{ab}c}$ of the Lie algebra $\mathfrak g$ given by $[T^a, T^b] = if^{ab}_{\phantom{ab}c}T^c$. We write $\phi = \phi_aT^a$. Then $\dfrac{\partial V(\phi,\phi^\dagger)}{\partial\phi^\dagger}$ is a matrix in $\mathfrak{su}(n, \mathbb C)$ and its $a$-th component (with respect to the basis $T^a$) is given by $\dfrac{\partial V(\phi,\phi^\dagger)}{\partial\phi_a^*}$. Since
$$
[\phi_aT^a,\phi_bT^b] = i\phi_a\phi_bf^{ab}_{\phantom{ab}d}T^d, \quad [T^d, T^c] = if^{dc}_{\phantom{dc}e}T^e,
$$
we write
\begin{align*}
\left[[\phi,\phi^\dagger],\phi\right] - v^2\phi &= \left[[\phi_aT^a,\phi_bT^b],\phi_cT^c\right] - v^2\phi_eT^e \\
&= -(f^{ab}_{\phantom{ab}d}f^{dc}_{\phantom{dc}e}\phi_a\phi_b\phi_c + v^2\phi_e)T^e.
\end{align*}
 Then we have

\begin{align*}
V(\phi,\phi^\dagger)&= {\rm Tr}\big( (\left[[\phi,\phi^\dagger],\phi\right]-v^2\phi)^\dagger(\left[[\phi,\phi^\dagger],\phi\right]-v^2\phi)\big) \\
&= (f^{ab}_{\phantom{ab}d}f^{dc}_{\phantom{dc}e}\phi_a^*\phi_b^*\phi_c^* + v^2\phi_e^*) (f^{a'b'}_{\phantom{a'b'}d'}f^{d'c'}_{\phantom{d'c'}e'}\phi_{a'}\phi_{b'}\phi_{c'} + v^2\phi_{e'}){\rm Tr}(T^e T^{e'}) \\
&= 2\sum_{e}(f^{ab}_{\phantom{ab}d}f^{dc}_{\phantom{dc}e}\phi_a^*\phi_b^*\phi_c^* + v^2\phi_e^*) (f^{a'b'}_{\phantom{a'b'}d'}f^{d'c'}_{\phantom{d'c'}e}\phi_{a'}\phi_{b'}\phi_{c'} + v^2\phi_{e}),
\end{align*}
 where we used ${\rm Tr}(T^e T^{e'}) = 2\delta^{ee'}$. Taking the partial derivative with respect to $\phi_a^*$, we have
\begin{align}\label{phi-a}
\dfrac{\partial V(\phi,\phi^\dagger)}{\partial\phi_a^*} = 2\sum_{e}(f^{ab}_{\phantom{ab}d}f^{dc}_{\phantom{dc}e}\phi_b^*\phi_c^* + v^2\delta^{ea}) (f^{a'b'}_{\phantom{a'b'}d'}f^{d'c'}_{\phantom{d'c'}e}\phi_{a'}\phi_{b'}\phi_{c'} + v^2\phi_{e}),
\end{align}
and hence we conclude that the essential terms in $\mathcal V(\phi,\phi^\dagger)$ are linear, cubic and quintic terms of $\phi$ and $\phi^\dagger$. The readers find more discussion on non-abelian gauge symmetry in \cite{sred, wein, yang}.

\section*{Acknowledgements}
Y. Cho was supported in part by NRF-2018R1D1A3B07047782(Republic of Korea). S. Hong was supported by NRF-2018R1D1A3B07047782, NRF-2018R1A2B2006298, and NRF-2016K2A9A2A13003815.


\end{document}